\documentclass{article}
\usepackage[sort&compress,longnamesfirst]{natbib}
\usepackage{amsthm} 
\usepackage{amssymb}
\usepackage{amsmath}
\usepackage{stmaryrd}
\usepackage{graphicx,xcolor}
\usepackage[all]{xy}
\usepackage[inline]{enumitem}
\usepackage{url}

\newcommand{\Topcat}{\mathbf{Top}}
\newcommand\PreTopcat{\mathbf{PreTop}}
\newcommand\adFrm{\mathbf{adFrm}}
\newcommand\Frm{\mathbf{Frm}}
\newcommand{\identity}[1]{\mathrm{id}_{#1}}
\newcommand\Term{{\mathbf 1}}
\newcommand\con{\mathsf{con}}
\newcommand\tot{\mathsf{tot}}
\newcommand\fof{\mathsf{sup}}
\newcommand\cou{\mathsf{sub}}
\newcommand\logleq{\unlhd}
\newcommand\logwedge{\mathbin{\overline{\wedge}}}
\newcommand\biglogwedge{\mathop{\overline{\bigwedge}}}
\newcommand\logvee{\mathbin{\overline{\vee}}}
\newcommand\biglogvee{\mathop{\overline{\bigvee}}}
\newcommand\ff{\overline{\mathit{ff}}}
\newcommand\tru{\overline{\mathit{tt}}}
\newcommand\eqdef{\mathrel{\buildrel \text{def}\over=}}
\newcommand\Open{\mathcal O}
\newcommand\Al{\mathcal A}
\newcommand\adO{\Open^{ad}}
\newcommand\pt{\textbf{pt}\,}
\newcommand\adpt{\textbf{pt}^{ad}\,}
\newcommand\OmegaL{\Omega\mskip-5mu L}
\newcommand\Ind{\textbf{Ind}}
\newcommand\Discr{\textbf{Discr}}
\newcommand\dsqcup{\bigsqcup\nolimits^{\scriptstyle\uparrow}}
\newcommand\dcup{\bigcup\nolimits^{\scriptstyle\uparrow}}

\newcommand\fcap{\bigcap\nolimits^{\scriptstyle\downarrow}}
\newcommand\upc{\mathop{\uparrow}}
\newcommand\dc{\mathop{\downarrow}}
\newcommand\diff{\smallsetminus}

\newtheorem{theorem}{Theorem}[section]
\newtheorem{proposition}[theorem]{Proposition}
\newtheorem{corollary}[theorem]{Corollary}
\newtheorem{lemma}[theorem]{Lemma}
\newtheorem{fact}[theorem]{Fact}
\newtheorem{remark}[theorem]{Remark}
\newtheorem{deflem}[theorem]{Definition and Lemma}
\newtheorem{definition}[theorem]{Definition}

\newcommand\up[1]{\textcolor{red}{#1}}
\newcommand\down[1]{\textcolor{blue}{#1}}
\newcommand\both[1]{\textcolor{purple}{#1}}

\DeclareRobustCommand{\VAN}[3]{#2} 

\author{Jean Goubault-Larrecq}
\title{Stone Duality for Preordered Topological Spaces}
\date{\today}

\begin{document}
\maketitle

\begin{abstract}
  A preordered topological space is a topological space with a
  preordering.  We exhibit a Stone-like duality for preordered
  topological spaces, inspired by a similar duality for bitopological
  spaces, due to Jung-Moshier and Jakl, and by a duality for
  preordered sets due to Bonsangue, Jacobs and Kok.
\end{abstract}

\section{Introduction}
\label{sec:introduction}

This work grew out of discussions with Nesta van der Schaaf on his
work on \emph{ordered locales} \citep{vdS:ordered:locales}, see also
\citet{HvdS:ordered:locales}.  Ordered locales are meant as a theory
of ordered topological spaces without points, in the same sense that
locale theory is a theory of topological spaces without points; the
novelty is the addition of a partial ordering.  Their strategy is to
equip locales with a partial ordering, and their main result is an
adjunction between the category of ordered topological spaces
satisfying a property that the authors call ``with open cones and
enough points'', and the category of ordered locales satisfying a
certain property ($\bullet$) \citep[Theorem~6.3]{HvdS:ordered:locales}.
The property ($\bullet$) is a bit awkward, as it mentions the points
of the locales, something that one usually tries to avoid in
specifying properties of locales.

This paper proposes another adjunction between (pre)ordered
topological spaces and another form of enrichment of the notion of
locales.  (Well, really three adjunctions, depending on whether we
want to be able about the future, the past, or both.)  We do not need
to restrict to any subcategory of (pre)ordered topological spaces, and
there is no condition mentioning points on the localic side.  Although
the definitions will appear to be complex, the general idea is pretty
simple, borrowing a notion of Stone duality for preordered \emph{sets}
from \citet{BJK:oframes} and combining it with another one for
bitopological spaces from \cite{JM:HM}, with improvements due to
\citet{Jakl:dframes}.  I will explain it in more detail in
Section~\ref{sec:outl-prel}.

There are several alternate proposals for dualities for certain
preordered topological spaces.  Notably, the category of compact
pospaces, with continuous order-preserving maps as morphisms, is
equivalent to the category of stably compact spaces, with perfect maps
as morphisms \citep{Nachbin:pospace,Nachbin:toporder}.  (A
\emph{pospace} is a topological space with a partial ordering whose
graph is closed in the product.  We will not define the other terms.)
The usual Stone adjunction $\Open \dashv \pt$ restricts to an
equivalence between the category of compact pospaces and the opposite
of the category of stably continuous frames, although one has to
beware that the morphisms are continuous maps, not perfect maps on the
first category.  There is a similar adjunction that works with perfect
maps and perfect locale morphisms (see \citet{Escardo:corefl:scomp}
and the references given therein, including earlier work by
\citet{BB:scframes}), and combining it with Nachbin's equivalence
gives a duality for compact pospaces.  The present approach does not
require compactness, and will work for all preordered topological
spaces, not just pospaces.

\section{Basic idea, preliminaries, and outline}
\label{sec:outl-prel}

We write $\Topcat$ for the category of topological spaces and
continuous maps.  Given any topological space $X$, its open subsets
form a frame $\Open X$.  A \emph{frame} is a complete lattice in which
arbitrary suprema distribute over binary infima.  A frame homomorphism
is a map between frames that preserves arbitrary suprema and finite
infima.  Frame together with frame homomorphisms form a category
$\Frm$.  Its opposite category $\Frm^{op}$ is the category of
\emph{locales}.  There is a functor
$\Open \colon \Topcat \to \Frm^{op}$.  We have already defined it on
objects.  Its action on morphisms $f \colon X \to Y$ is given by
$\Open f \eqdef f^{-1}$.  In the converse direction, there is a
functor $\pt \colon \Frm^{op} \to \Topcat$ that is right-adjoint to
$\Open$---in notation, $\Open \dashv \pt$.  This is (a modern reading
of) \emph{Stone duality}, see \citet[Section~8.1]{JGL-topology}.

A preordered topological space $(X, \leq)$ is a topological space $X$
with a preordering $\leq$.  We will also simply write $X$ for
$(X, \leq)$, and $\leq_X$ for its preordering.  We propose a form a
Stone duality for such spaces, based on the following ideas.

A preordering on a set is described equivalently as its collection of
upwards-closed subsets, which is an Alexandroff topology, namely a
topology in which arbitrary intersections of open sets are open.  A
Stone-like duality for such space was described by
\citet[Section~6.4]{BJK:oframes}: the duals of Alexandroff spaces
proposed there is given by completely distributive lattices, with
complete lattice homomorphisms as morphisms.  (See Lemma~6.5 in
\citep{BJK:oframes}, realizing that what the authors call complete
lattices generated by their M-prime elements are the same thing as
completely distributive lattices.)  A \emph{completely distributive
  lattice} is a complete lattice in which
$\bigwedge_{i \in I} \bigvee_{j \in J_i} x^i_j = \bigvee_{f \in
  \prod_{i \in I} J_i} \bigwedge_{i \in I} x^i_{f (i)}$ for every
family ${(x^i_j)}_{i \in I, j \in J_i}$, where $I$ and $J_i$ are
arbitrary sets.  $\prod_{i \in I} J_i$ is the space of all maps that
send every element $i \in I$ to an element of $J_i$.  The notations
$\vee$ and $\bigvee$ denote suprema, and $\wedge$ and $\bigwedge$
denote infima.  Complete lattice homomorphisms preserve all suprema
and all infima.

With the viewpoint of preordered sets as spaces with an Alexandroff
topology, a preordered topological space is a special kind of
bitopological space, namely a space with two topologies.  Stone-like
duals of bitopological spaces were proposed by several authors, and
notably by \citet{JM:HM} and by \citet{Jakl:dframes}: taking Jakl's
way of presenting them, those are \emph{d-frames}, namely a pair of
frames linked by two relations $\con$ (consistency) and $\tot$
(totality) satisfying certain conditions.  We will mix the two
approaches.

There are several variants to this.  The second topology, which should
be the Alexandroff topology of the preordering $\leq$, can be
complemented with a third topology, namely the Alexandroff topology of
the opposite preordering $\geq$.  We will show \up{in red} what we
need if we use the second topology but not necessarily the third,
\down{in blue} what happens if we use the third topology but not
necessarily the second, and \both{in purple} what happens if we
include both, and their interactions.  Hence you can read this paper
with goggles that will let you read only ordinary text and red text,
or only ordinary text and blue text, or only ordinary, red and blue
text, or everything, and you will get one theory for unordered spaces
plus three different theories for ordered spaces, all dealt with at
the same time.  One may think of \up{red} text as a theory of looking
into the \up{future}, \down{blue} text as a theory of looking into the
\down{past}, and \both{purple} as a theory that combines the two.

We will rely on standard references for concepts in order theory and
topology \cite{JGL-topology,GHKLMS:contlatt,AJ:domains}.  A
\emph{poset} is a partially ordered set.  A \emph{directed family} in
a poset $P$ is a non-empty family ${(x_i)}_{i \in I}$ such that any
two elements $x_i$, $x_j$ have a common upper bound $x_k$.  A
\emph{downwards-closed} subset $A$ in $P$ is one such that $y \in A$
and $x \leq y$ imply $x \in A$, where $\leq$ is the ordering of $P$.
A \emph{Scott-closed} subset of $P$ is one that is downwards-closed
and closed under suprema of directed families.  An
\emph{upwards-closed} subset $A$ in $P$ is one such that $x \in A$ and
$x \leq y$ imply $y \in A$.  When we deal with several partial
orderings on the same set $P$, we will explicitly mention it, as in
$\leq$-downwards-closed or $\leq$-Scott-closed.

The plan of the paper is as follows.  We implement our strategy
described above and define ad-frames in Section~\ref{sec:ad-frames}.
We show that there is an adjunction $\adO \dashv \adpt$ between
preordered topological spaces and ad-frames in
Section~\ref{sec:ado-dashv-adpt}.  Analogously to sobrification, a
natural construction arising from the usual Stone duality adjunction
$\Open \dashv \pt$, we explore ad-sobrification in
Section~\ref{sec:ad-sobrification}.  We compare it to sobrification in
Section~\ref{sec:sobr-ad-sobr}.  Just like $\Open \dashv \pt$, we show
that the adjunction $\adO \dashv \adpt$ adjunction is idempotent in
Section~\ref{sec:ad-sobr-adjunct}.  One selling point of van der
Schaaf's notion of ordered locales is that his adjunction lifts the
adjunction $\Open \dashv \pt$.  We show that a similar situation
occurs with $\adO \dashv \adpt$ in Section~\ref{sec:ado-dashv-adpt-1}.

\section{Ad-frames}
\label{sec:ad-frames}

Following the strategy set forth in Section~\ref{sec:outl-prel}, we
define an \emph{ad-frame} as a tuple
$(\Omega, L, \up{\tot, \con}, \down{\fof, \cou})$ where $\Omega$ is a
frame, $L$ is a completely distributive lattice, and
$\tot, \con \subseteq \Omega \times L$ are binary relations satisfying
certain properties listed below.  
We will sometimes call \up{$\tot$ the totality relation, $\con$ the
  consistency relation}, \down{$\fof$ the containment relation, and
  $\cou$ the inclusion relation}.  In order to state them, we need to
introduce the following notation, inspired from \citet{Jakl:dframes}:
\begin{itemize}
\item we write $\leq$ for the ordering on both $\Omega$ and $L$,
  $\wedge$ for binary infima, $\bigwedge$ for arbitrary infima, $\top$
  for their top element, $\vee$ for binary suprema, $\bigvee$ for
  arbitrary suprema, $\bot$ for their bottom element; we also write
  $\geq$ for the opposite of $\leq$;
\item we write $\sqsubseteq$ for the ordering $\leq \times \leq$ on
  $\Omega \times L$ (the ``information ordering''); the corresponding
  infima and  suprema are written as $\sqcap$, $\bigsqcap$, $\sqcup$,
  $\bigsqcup$; the top element is $(\top, \top)$ and the bottom
  element is $(\bot, \bot)$;
\item we write $\logleq$ for $\leq \times \geq$ (the ``logical
  ordering'') on $\Omega \times L$; the corresponding infima and
  suprema are written as $\logwedge$, $\biglogwedge$, $\logvee$,
  $\biglogvee$; the top element is $\tru \eqdef (\top, \bot)$ and the
  bottom element is $\ff \eqdef (\bot, \top)$.
\end{itemize}
With those notations, we make the following definition.
\begin{definition}
  \label{defn:adframe}
  A \emph{ad-frame} is a tuple
  $(\Omega, L, \up{\tot, \con}, \down{\fof, \cou})$ where $\Omega$ is
  a frame, $L$ is a completely distributive lattice, and:
  \begin{itemize}
  \item \up{$\tot$ is an $\sqsubseteq$-upwards-closed subset of
      $\Omega \times L$ containing $\ff$, $\tru$ and closed under
      $\logwedge$ and $\biglogvee$;}
  \item \up{$\con$ is a $\sqsubseteq$-Scott-closed subset of
      $\Omega \times L$ containing $\ff$, $\tru$ and closed under
      $\logwedge$ and $\biglogvee$;}
  \item \down{$\fof$ is an $\logleq$-upwards-closed subset of
      $\Omega \times L$ containing $(\bot, \bot)$, $(\top, \top)$ and
      closed under $\sqcap$ and $\bigsqcup$;}
  \item \down{$\cou$ is a $\logleq$-Scott-closed subset of
      $\Omega \times L$ containing $(\bot, \bot)$, $(\top, \top)$ and
      closed under $\sqcap$ and $\bigsqcup$;}
  \item the following \emph{interaction laws} are satisfied:
    \begin{itemize}
    \item \up{for all $(u, a) \in \con$ and $(v, b) \in \tot$ such
        that $u=v$ or $a=b$, $(u, a) \sqsubseteq (v, b)$;}
    \item \down{forall $(u, a) \in \cou$ and $(v, b) \in \fof$ such
        that $u=v$ or $a=b$, $(u, a) \logleq (v, b)$;}
    \item \both{for every $(u, a) \in \con \cap \cou$, $u = \bot$;}
    \item \both{for every $(v, b) \in \tot \cap \fof$, $v = \top$;}
    \item \both{for every $(v, b) \in \con \cap \fof$, $b = \bot$;}
    \item \both{for every $(u, a) \in \tot \cap \cou$, $a = \top$.}
    \end{itemize}
  \end{itemize}
\end{definition}
The best way to understand Definition~\ref{defn:adframe} is by looking
at the following canonical example of ad-frames.  More generally, it
is recommended to understand an ad-frame as a form of abstraction of a
preordered topological space, where $\Omega$, $M$, $\up{\tot}$,
$\up{\con}$, $\down{\fof}$, $\down{\cou}$ are to be understood as
follows.

\begin{deflem}
  \label{deflem:adO}
  For every preordered topological space $X$, we define $\adO X$ as
  $(\Omega, L, \up{\tot, \con}, \down{\fof, \cou})$ where:
  \begin{itemize}
  \item $\Omega$ is the lattice $\Open X$ of open subsets of $X$;
  \item $L$ is the collection of $\leq_X$ -upwards-closed subsets of
    $X$;
  \item \up{for every $(U, A) \in \Omega \times L$, $(U, A) \in \tot$ if
      and only if $U \cup A = X$;}
  \item \up{for every $(U, A) \in \Omega \times L$, $(U, A) \in \con$ if
      and only if $U \cap A = \emptyset$;}
  \item \down{for every $(U, A) \in \Omega \times L$, $(U, A) \in
      \fof$ if and only if $U \supseteq A$;}
  \item \down{for every $(U, A) \in \Omega \times L$, $(U, A) \in
      \cou$ if and only if $U \subseteq A$.}
  \end{itemize}
  Then $\adO X$ is an ad-frame.
\end{deflem}
Before we proceed, let us observe that $(U, A) \in \fof$ if and only
if $U \cup \neg A = X$ and $(U, A) \in \cou$ if and only if
$U \cap \neg A = X$, where $\neg A$ denotes the complement of $A$ in
$X$.  While the elements $A \in L$ are the $\leq_X$-upwards-closed
subsets of $X$, the elements $\neg A$ where $A \in L$ are the
$\leq_X$-downwards-closed subsets of $X$, and $\fof$ and $\cou$ are
the relations between $\Omega$ and
$\{\neg A \mid A \in L\} \cong L^{op}$ that mimic $\tot$ and $\con$
between $\Omega$ and $L$.

\begin{proof}
  $\Omega$ is obviously a frame, and it is easy to see that $L$ is a
  completely distributive lattice.  Explicitly, in $L$ suprema are
  unions and infima (not just finite infima) are intersections.  Then
  complete distributivity boils down to the equality
  $\bigcap_{i \in I} \bigcup_{j \in J_i} A^i_j = \bigcup_{f \in
    \prod_{i \in I} J_i} \bigcap_{i \in I} A^i_{f (i)}$, where $A^i_j$
  is $\leq_X$-upwards-closed, and this is a classic instance of the
  axiom of choice.

  Let us take the notations $\sqsubseteq$, $\logleq$, etc.\ as given
  above.

  \up{It is easy to see that $\tot$ is $\sqsubseteq$-upwards closed:
    if $U \cup A = X$ and $U \subseteq V$, $A \subseteq B$, then
    $V \cup B = X$.  It contains $\ff = (\emptyset, X)$ and
    $\tt = (X, \emptyset)$.  For all $(U, A), (V, B) \in \tot$,
    $(U \cap V) \cup (A \cup B) = (U \cup A \cup B) \cap (V \cup A
    \cup B) = X \cap X = X$, so
    $(U, A) \logwedge (V, B) = (U \cap V, A \cup B) \in \tot$.  For
    every family of pairs $(U_i, A_i) \in \tot$, where $i \in I$, we
    claim that
    $\biglogvee_{i \in I} (U_i, A_i) = (\bigcup_{i \in I} U_i,
    \bigcap_{i \in I} A_i)$ is in $\tot$.  For every $x \in X$, either
    $x \in \bigcup_{i \in I} U_i$, or $x$ is in no $U_i$; since
    $(U_i, A_i) \in \tot$, $x$ must be in $A_i$, and this for every
    $i \in I$; hence $x \in \bigcap_{i \in I} A_i$.}

  \down{The fact that $\fof$ is $\logleq$-upwards closed, that it
    contains $(\bot, \bot)$ and $(\top, \top)$ and is closed under
    $\sqcap$ and $\bigsqcup$ is similar, up to replacing $A$ by
    $\neg A$ for every $A \in L$.  Explicitly, if $(U, A) \in \fof$
    (namely $U \supseteq A$) and $U \subseteq V$ and $A \supseteq B$,
    then $V \supseteq B$.  $\bot$ ($=\emptyset$) contains $\bot$, so
    $(\bot, \bot) \in \fof$ and similarly $(\top, \top) \in \fof$
    (where $\top=X$).  If $U \supseteq A$ and $V \supseteq B$ then
    $U \cap V \supseteq A \cap B$, so $\fof$ is closed under $\sqcap$.
    If $U_i \supseteq A_i$ for every $i \in I$, then
    $\bigcup_{i \in I} U_i \supseteq \bigcup_{i \in I} A_i$, so $\fof$
    is closed under $\bigsqcup$.}
  
  \up{Let us check that $\con$ is $\sqsubseteq$-Scott-closed.  If
    $(U, A) \sqsubseteq (V, B)$, namely if $U \subseteq V$ and
    $A \subseteq B$, and if $(V, B) \in \con$, namely if $V$ and $B$
    are disjoint, then so are $U$ and $A$.  Hence $\con$ is
    downwards-closed.  Let $(U_i, A_i)$ be a $\sqsubseteq$-directed
    family in $\con$.  Hence $U_i$ and $A_i$ are disjoint for every
    $i \in I$.  We observe that
    $\dsqcup_{i \in I} (U_i, A_i) = (\dcup_{i \in I} U_i, \dcup_{i \in
      I} A_i)$ must be in $\con$, namely that $\dcup_{i \in I} U_i$
    and $\dcup_{i \in I} A_i$ cannot intersect.  Otherwise, let $x$ be
    in the intersection.  Then $x \in U_i$ and $x \in A_j$ for some
    $i, j \in I$.  By directedness, we can find $k \in I$ such that
    $(U_i, A_i), (U_j, A_j) \sqsubseteq (U_k, A_k)$; in particular,
    $x \in U_k$ and $x \in A_k$, which is impossible since $U_k$ and
    $A_k$ are disjoint.}

  \down{Similarly, $\cou$ is $\logleq$-Scott-closed.  If
    $(U, A) \logleq (V, B)$, namely if $U \subseteq V$ and
    $A \supseteq B$, and if $(V, B) \in \cou$, namely if
    $V \subseteq B$, then $U \subseteq A$, so $\cou$ is
    $\logleq$-downwards-closed.  For every directed family
    ${(U_i, A_i)}_{i \in I}$ in $\cou$ (namely, $U_i \subseteq A_i$
    for every $i \in I$), its supremum with respect to $\logleq$ is
    $(U, A)$ where $U \eqdef \dcup_{i \in I} U_i$ and
    $A \eqdef \fcap_{i \in I} A_i$, and then $U \subseteq A$.  Indeed,
    for every $x \in U$, for every $i \in I$, it suffices to show that
    $x \in A_i$.  Since $x \in U$, $x \in U_j$ for some $j \in I$.  By
    directedness, there is a $k \in I$ such that
    $(U_i, A_i), (U_j, A_j) \logleq (U_k, A_k)$.  In particular
    $U_j \subseteq U_k$, so $x \in U_k$, and $A_i \supseteq A_k$; but
    $U_k \subseteq A_k$, so $x \in A_i$.}

  \up{The relation $\con$ contains $\ff = (\emptyset, X)$ and
    $\tt = (X, \emptyset)$.  For all $(U, A), (V, B) \in \con$,
    $U \cap A = \emptyset$ and $V \cap B = \emptyset$, so
    $(U \cap V) \cap (A \cup B) = (U \cap V \cap A) \cup (U \cap V
    \cap B) = \emptyset \cup \emptyset = \emptyset$, showing that
    $(U, A) \logwedge (V, B) = (U \cap V, A \cup B)$ is in $\con$.
    For every family $(U_i, A_i)$ of elements of $\con$, where
    $i \in I$, we claim that
    $\bigsqcup_{i \in I} (U_i, A_i) = (\bigcup_{i \in I} U_i,
    \bigcap_{i \in I} A_i)$ is in $\con$.  Otherwise, there would be a
    point $x \in \bigcap_{i \in I} A_i$ that is in
    $\bigcup_{i \in I} U_i$, hence in some $U_i$.  But since
    $(U_i, A_i) \in \con$, $x$ cannot be in $A_i$ after all, which is
    impossible.}

  \down{The relation $\cou$ contains
    $(\bot, \bot) = (\emptyset, \emptyset)$ and
    $(\top, \top) = (X, X)$, since $\emptyset \supseteq \emptyset$ and
    $X \supseteq X$.  For all $(U, A), (V, B) \in \cou$, we have
    $U \subseteq A$ and $V \subseteq B$, so
    $U \cap V \subseteq A \cap B$, showing that $\cou$ is closed under
    $\sqcap$.  For every family of pairs $(U_i, A_i) \in \cou$,
    $i \in I$,
    $\bigcup_{i \in I} U_i \subseteq \bigcup_{i \in I} A_i$, so $\cou$
    is closed under $\bigsqcup$.}
  
  Finally, we check the interaction laws.
  \begin{itemize}
  \item \up{Let $(U, A) \in \con$ and $(V, B) \in \tot$.  Hence $U$ is
      disjoint from $A$ and $V \cup B = X$.  If $U=V$, then every
      $x \in A$ is outside $U$, hence outside $V$, hence must be in
      $B$; so $A \subseteq B$ and therefore
      $(U, A) \sqsubseteq (V, B)$.  If $A=B$, then every $x \in U$ is
      outside $A$, hence outside $B$, hence must be in $V$; so
      $U \subseteq V$ and therefore $(U, A) \sqsubseteq (V, B)$.}
  \item \down{Let $(U, A) \in \cou$ and $(V, B) \in \fof$.  Hence
      $U \subseteq A$ and $V \supseteq B$.  If $U=V$, then
      $A \supseteq B$, so $(U, A) \logleq (V, B)$.  If $A=B$, then
      $U \subseteq V$, so $(U, A) \logleq (V, B)$.}
  \item \both{Let $(U, A) \in \con \cap \cou$.  Then $U \cap A =
      \emptyset$ and $U \subseteq A$, so $U = \emptyset$.}
  \item \both{Let $(V, B) \in \tot \cap \fof$.  Then $V \cup B = X$
      and $V \supseteq B$, so $V=X$.}
  \item \both{Let $(V, B) \in \con \cap \fof$.  Then $V \cap B = \emptyset$
      and $V \supseteq B$, so $B = \emptyset$.}
  \item \both{Let $(U, A) \in \tot \cap \cou$.  Then $U \cup A = X$
      and $U \subseteq A$, so $A = X$.}
  \end{itemize}
\end{proof}

\begin{definition}
  \label{defn:ad:mor}
  An \emph{ad-frame homomorphism} from an ad-frame
  $(\Omega, L, \up{\tot, \con}, \allowbreak \down{\fof, \cou})$ to an
  ad-frame $(\Omega', L', \up{\tot', \con'}, \down{\fof', \cou'})$ is
  a pair $(\varphi, p)$ where:
  \begin{itemize}
  \item $\varphi$ is a frame homomorphism from $\Omega$ to $\Omega'$,
    namely $\varphi$ preserves arbitrary suprema and finite infima;
  \item $p$ is a complete lattice homomorphism from $L$ to $L'$,
    namely preserves arbitrary suprema and arbitrary infima;
  \item \up{the pair $(\varphi, p)$ preserves totality: for every
      $(u, a) \in \tot$, $(\varphi (u), p (a)) \in \tot'$;}
  \item \up{the pair $(\varphi, p)$ preserves consistency: for every
      $(u, a) \in \con$, $(\varphi (u), p (a)) \in \con'$;}
  \item \down{the pair $(\varphi, p)$ preserves containment: for every
      $(u, a) \in \fof$, $(\varphi (u), p (a)) \in \fof'$;}
  \item \down{the pair $(\varphi, p)$ preserves inclusion: for every
      $(u, a) \in \cou$, $(\varphi (u), p (a)) \in \cou'$.}
  \end{itemize}
\end{definition}
Ad-frames and ad-frame homomorphisms form a category that we will
write as $\adFrm$.

Preordered topological spaces form a category $\PreTopcat$, whose
objects are the preordered topological spaces and whose morphisms are
the continuous preorder-preserving maps---continuous with respect to
the underlying topological spaces, and preserving the preorder, namely
monotonic with respect to the given preorderings.
\begin{lemma}
  \label{lemma:ad:mor}
  For every continuous, order-preserving map $f \colon X \to Y$
  between preordered topological spaces, the pair
  $\adO f \eqdef (f^{-1}, f^{-1})$ is an ad-frame homomorphism from
  $\adO Y$ to $\adO X$.  Here $f^{-1}$ denotes the function that maps
  every subset (whether open or $\leq_Y$-upwards-closed) of $Y$ to its
  inverse image under $f$.
\end{lemma}
\begin{proof}
  Let $(\varphi, p) \eqdef \adO f$.  It is clear that $\varphi$ is a
  frame homomorphism and that $p$ is a complete lattice homomorphism.
  \up{For every total pair $(V, B)$ in $\adO Y$, $V \cup B = Y$, so
    $\varphi (V) \cup \varphi (B) = f^{-1} (V \cup B) = f^{-1} (Y)=X$.
    For every consistent pair $(V, B)$ in $\adO Y$,
    $V \cap B = \emptyset$, so
    $\varphi (V) \cap \varphi (B) = f^{-1} (V \cap B) = f^{-1}
    (\emptyset) = \emptyset$.}  \down{For every pair $(V, B)$ such
    that $V \subseteq B$, $f^{-1} (V) \subseteq f^{-1} (B)$, and
    similarly with $\supseteq$.}
\end{proof}

\begin{corollary}
  \label{corl:adO:func}
  $\adO$ is a functor from $\PreTopcat$ to $\adFrm^{op}$.
\end{corollary}

\section{The $\adO \dashv \adpt$ adjunction}
\label{sec:ado-dashv-adpt}

The terminal object $\Term$ of $\PreTopcat$ is the one-element space
$\{*\}$ with the unique topology that one can put on it, and with the
only preordering that one can put on it.  Then both $\Open \Term$ and
the lattice of $\leq_\Term$-upwards-closed subsets of $\Term$ are
$\{\emptyset, \Term\}$, with $\emptyset$ below $\Term$.  It follows
that $\adO \Term$ is the four-element lattice
$\{\emptyset, \Term\}^2$, and we check easily that \up{all its pairs
  are total except $(\emptyset, \emptyset)$, that all its pairs are
  consistent except $(\Term, \Term)$}, \down{that all its pairs are in
  containment except $\ff = (\emptyset, \Term)$, and that al its pairs
  are in the inclusion relation except $\ff = (\Term, \emptyset)$.}

An ad-frame homomorphism $(\varphi, p)$ from an ad-frame
$(\Omega, L, \tot, \con)$ to $\adO \Term$ is entirely characterized by
$x \eqdef \varphi^{-1} (\Term)$ and by $s \eqdef p^{-1} (\Term)$.
Since $\varphi$ is a frame homomorphism, $x$ must be a completely
prime filter of elements of $\Omega$.  (A \emph{filter} of elements of
$\Omega$ is a subset $F$ of $\Omega$ that is upwards-closed and closed
under finite infima.  It is \emph{completely prime} if and only if its
complement is closed under arbitrary suprema, if and only if every
family ${(x_i)}_{i \in I}$ whose supremum is in $F$ contains some
$x_i$ in $F$.)  Since $p$ is a complete lattice homomorphism, $s$ must
be a completely prime \emph{complete} filter of elements of $L$, where
a complete filter is any subset of $L$ that is upwards-closed and
closed under arbitrary infima (not just finite infima).  The fact that
$(\varphi, p)$ must preserve \up{totality, consistency},
\down{containment and inclusion} reflects into the following
definition.
\begin{definition}
  \label{defn:adpt}
  A \emph{point} of an ad-frame
  $\OmegaL \eqdef (\Omega, L, \up{\tot, \con}, \down{\fof, \cou})$ is
  a pair $(x, s)$ where:
  \begin{itemize}
  \item $x$ is a completely prime filter of elements of $\Omega$;
  \item $s$ is a completely prime complete filter of elements of $L$;
  \item \up{every pair $(u, a) \in \tot$ is such that $u \in x$ or
      $a \in s$;}
  \item \up{every pair $(u, a) \in \con$ is such that $u \not\in x$ or
      $a \not\in s$;}
  \item \down{every pair $(u, a) \in \fof$ is such that $u \in x$ or
      $a \not\in s$;}
  \item \down{every pair $(u, a) \in \cou$ is such that $u \not\in x$
      or $a \in s$.}
  \end{itemize}
  Let $\adpt \OmegaL$ be the preordered topological space $X$:
  \begin{itemize}
  \item whose elements are the points of $\OmegaL$,
  \item whose open sets are the sets $\Open_u \eqdef \{(x, s) \in
    \adpt \OmegaL \mid u \in x\}$, where $u$ ranges over $\Omega$,
  \item and whose preordering $\leq_X$ is given by $(x, s) \leq_X (y,
    t)$ if and only if $s \subseteq t$.
  \end{itemize}
  We also write $\Al_a$ for $\{(x, s) \in \adpt \OmegaL \mid a \in
  s\}$ for every $a \in L$.
\end{definition}

\begin{remark}
  \label{rem:adpt:alt}
  (Can be skipped on first reading.)  It is well-known that a
  completely prime filter $x$ of elements of $\Omega$ is exactly the
  collection $\{u \in \Omega \mid u \not\leq p\}$ for some unique
  prime element $p$ of $\Omega$ ($p$ is prime if and only if
  $p \neq \top$ and for all $u, v \in \Omega$ such that
  $u \wedge v \leq p$, then $u \leq p$ or $v \leq p$).  A completely
  prime complete filter $s$ of elements of $L$ can similarly be
  written as $\{a \in L \mid a \not\leq q\}$ for some unique
  completely prime element $q$ of $L$ ($q$ is \emph{completely prime}
  if and only if for every family ${(a_i)}_{i \in I}$ in $L$ such that
  $\bigwedge_{i \in I} a_i \leq q$, then $a_i \leq q$ for some
  $i \in I$).  Alternatively, $s$ can also be expressed as $\upc b$
  for some unique completely coprime element $b$ of $L$ ($b$ is
  \emph{completely coprime} if and only if for every family
  ${(a_i)}_{i \in I}$ in $L$ such that $b \leq \bigvee_{i \in I} a_i$,
  then $b \leq a_i$ for some $i \in I$).  The relation between $q$ and
  $b$ is that for every $a \in L$, $b \leq a$ if and only if
  $a \not\leq q$ (if and only if $a \in s$).  For any two elements
  $q, b \in L$ (not just completely prime/completely coprime
  elements), let us write $q \pitchfork b$ if and only if every
  $a \in L$, $b \leq a$ if and only if $a \not\leq q$.  Then:
  \begin{itemize}
  \item for every $q \in L$, there at most one $b \in L$ such that $q
    \pitchfork b$: indeed, if there were two of them, $b$ and $b'$,
    then $b \leq a$ would be equivalent to $b' \leq a$ (and to $a
    \not\leq q$) for every $a \in L$, so $b \leq b'$ and $b' \leq b$,
    whence $b=b'$;
  \item symmetrically, for every $b \in L$, there is at most one $q
    \in L$ such that $q \pitchfork b$;
  \item if $q \pitchfork b$, then $q$ is completely prime and $b$ is
    completely coprime; indeed, if for every family
    ${(a_i)}_{i \in I}$ in $L$, if $a_i \leq q$ for no $i \in I$, then
    $b \leq a_i$ for every $i \in I$, hence
    $b \leq \bigwedge_{i \in I} a_i$, so
    $\bigwedge_{i \in I} a_i \not\leq q$; and similarly for showing
    that $b$ is completely coprime.
  \end{itemize}
  Then a point can be specified by a prime element $p$ of $\Omega$,
  and pair of elements $q \pitchfork b$ in $L$, under the proviso that
  \up{every pair $(u, a) \in \tot$ is such that $u \not\leq p$ or
    $b \leq a$, every pair $(u, a) \in \con$ is such that $u \leq p$
    or $b \not\leq a$},
  \down{every pair $(u, a) \in \fof$ is such that $u
    \not\leq p$ or $a \leq q$, and every pair $(u, a) \in \cou$ is
    such that $u \leq p$ or $a \not\leq q$.}
  Using the fact that \up{$\tot$ is $\sqsubseteq$-upwards-closed, that
    $\con$ is $\sqsubseteq$-downwards-closed}, \down{that $\fof$ is
    $\logleq$-upwards-closed, that $\cou$ is
    $\logleq$-downwards-closed}, these conditions simplify to:
  \begin{itemize}
  \item \up {for every $a \in L$, if $(p, a) \in \tot$ then
      $b \leq a$, and for every $u \in \Omega$, if $(u, b) \in \con$
      then $u \leq p$};
  \item \down{for every $a \in L$, if $(p, a) \in \fof$ then
      $a \leq q$, and for every $u \in \Omega$, if $(u, q) \in \cou$
      then $u \leq p$}.
  \end{itemize}
\end{remark}

\begin{remark}
  \label{rem:semiclosed}
  A preordered topological space $X$ is \emph{upper semi-closed}
  (resp., \emph{lower semi-closed}) if and only the upward downward
  closure $\upc_X x$ (resp., the upward closure $\dc_X x$) of every
  point $x$ with respect to $\leq_X$ is closed in $X$.  It is
  \emph{semi-closed} if and only if it is both upper and lower
  semi-closed.  Let us consider the following properties for an
  ad-frame
  $\OmegaL \eqdef (\Omega, L, \up{\tot, \con}, \down{\fof, \cou})$:
  \begin{description}
  \item[(usc)] \up{every element of $L$ is a supremum of elements of $\{a
    \mid (u, a) \in \tot \cap \con\}$;}
\item[(lsc)] \down{every element of $L$ is an infimum of elements of
    $\{a \mid (u, a) \in \fof \cap \cou\}$.}
  \end{description}
  When $\OmegaL$ is of the form $\adO X$ for some preordered
  topological space $X$, we have:
  \begin{itemize}
  \item \up{$X$ is upper semi-closed if and only if $\adO X$ satisfies
    (usc);}
  \item \down{$X$ is lower semi-closed if and only if $\adO X$ satisfies
    (lsc).}
  \end{itemize}
  \up{Indeed, the elements of $\{A \mid (U, A) \in \tot \cap \con\}$
    are exactly the $\leq_X$-upwards-closed, closed subsets of $X$.
    If $X$ is upper semi-closed, then every upwards-closed subset
    $A \in L$ is the union of the sets $\upc_X x$, $x \in A$, which
    are all in $\{A \mid (U, A) \in \tot \cap \con\}$.  Conversely, if
    every element of $L$ is a supremum of elements of
    $\{A \mid (U, A) \in \tot \cap \con\}$, then this holds in
    particular of subsets of the form $\upc_X x$, $x \in X$.  Then
    there must be a pair $(U, A) \in \tot \cap \con$ such that
    $A \subseteq \upc_X x$ and $x \in A$, and this forces
    $A = \upc_X x$; since $(U, A) \in \tot \cap \con$, $A$ is closed.
    Therefore $X$ is upper semi-closed.}  \down{If $X$ is lower
    semi-closed, then for every upwards-closed subset $A \in L$, its
    complement $X \diff A$ is the union of the sets $\dc_X x$,
    $x \in X \diff A$.  Hence $A$ is the intersection of the sets
    $X \diff \dc_X x$, $x \in X \diff A$.  All those sets are open and
    upwards-closed, hence in $\{A \mid (U, A) \in \fof \cap \cou\}$,
    since $(U, A) \in \fof \cap \cou$ if and only if $U=A$.
    Conversely, if every element of $L$ is an infimum of elements of
    $\{A \mid (U, A) \in \fof \cap \cou\}$, then in particular
    $X \diff \dc_X x$ is an intersection of open upwards-closed sets
    for every point $x \in X$.  Therefore $\dc_X x$ is a union of
    closed downwards-closed sets.  One of them---let us call it
    $C$---will contain $x$.  Since $C$ is downward-closed, it must
    contain $\dc_X x$, and by construction $C \subseteq \dc_X x$; so
    $\dc_X x = C$ is closed, showing that $X$ is lower semi-closed.}
  \both{In particular, $X$ is semi-closed if and only if $\adO X$
    satisfies both (usc) and (lsc).}
\end{remark}

\begin{lemma}
  \label{lemma:adpt}
  For every ad-frame
  $\OmegaL \eqdef (\Omega, L, \up{\tot, \con}, \down{\fof, \cou})$,
  $\adpt \OmegaL$ is a preordered topological space.  Its
  upwards-closed subsets are exactly the sets $\Al_a$, where
  $a \in L$.  The map $u \mapsto \Open_u$ is a surjective frame
  homomorphism from $\Omega$ to the lattice $\Open {\adpt \OmegaL}$ of
  open subsets of $\adpt \OmegaL$, and the map $a \mapsto \Al_a$ is a
  surjective complete lattice homomorphism from $L$ to the lattice of
  upwards-closed subsets of $\adpt L$.
\end{lemma}
\begin{proof}
  It is easy to see that $\Open_\top = X$,
  $\Open_u \cap \Open_v = \Open_{u \wedge v}$ for all
  $u, v \in \Omega$ (because every $x$ in a point $(x, s)$ of
  $\OmegaL$ is a filter),
  $\bigcup_{i \in I} \Open_{u_i} = \Open_{\bigvee_{i \in I} u_i}$ for
  every family ${(u_i)}_{i \in I}$ of elements of $\Omega$ (because
  every $x$ in a point $(x, s)$ of $\OmegaL$ is completely prime).  It
  follows that the collection of sets $\Open_u$, where $u \in \Omega$,
  is a topology on $\adpt \OmegaL$, and that $u \mapsto \Open_u$ is a
  surjective frame homomorphism.

  Let us write $X$ for $\adpt \OmegaL$, so that we have the notations
  $\leq_X$, $\dc_X$, $\upc_X$ at our disposal.

  Similarly, $a \mapsto \Al_a$ is a surjective complete lattice
  homomorphism from $L$ to the complete lattice
  $\Al \eqdef \{\Al_a \mid a \in L\}$.  Every element $\Al_a$ of $\Al$
  is upwards-closed with respect to $\leq_X$: for every
  $(x, s) \in \Al_a$, namely $a \in s$, for every point
  $(y, t) \in \adpt L$ such that $(x, s) \leq_X (y, t)$, namely such
  that $s \subseteq t$, we have $a \in s \subseteq t$, hence
  $(y, t) \in \Al_a$.  Conversely, we wish to show that every upwards
  closed subset $A$ of $X = \adpt L$ with respect to $\leq_X$ is of
  the form $\Al_a$ for some $a \in L$.  We start with the special case
  where $A$ is the upward closure $\upc_X (x, s)$ of a single point
  $(x, s) \in X$.  Let $a \eqdef \bigwedge s$.  Since $s$ is a
  complete filter, $a$ is in $s$, and is therefore the least element
  of $s$; it follows that $s$ is the upward closure of $a$ in $L$.
  Then $\Al_a$ is the collection of points $(y, t)$ such that
  $a \in t$, or equivalently such that $s \subseteq t$ (since $t$ is
  upwards closed in $L$), namely $\Al_a = \upc_X (x, s)$.  In the
  general case, let $A$ be any upwards closed subset of $X$ with
  respect to $\leq_X$.  Then $A$ is the union of the sets $\upc_X x$,
  where $x$ ranges over $A$.  Each set $\upc_X x$ can be written as
  $\Al_{a_x}$ for some $a_x \in L$.  Hence
  $A = \bigcup_{x \in A} \Al_{a_x} = \Al_{\bigvee_{x \in A} a_x}$.
\end{proof}

\begin{deflem}
  \label{deflem:eta}
  For every preordered space $X$, there is a continuous,
  order-preserving map $\eta_X \colon X \to \adpt {\adO X}$, which
  maps every $x \in X$ to the pair $(\mathcal N_x, \mathcal U_x)$,
  where $\mathcal N_x$ is the set of open neighborhoods of $x$ in $X$,
  and $\mathcal U_x$ is the collection of $\leq_X$-upwards-closed
  subsets of $X$ that contain $x$.
\end{deflem}
\begin{proof}
  We check that $(\mathcal N_x, \mathcal U_x)$ is a point of $\adO X$.
  It is clear that $\mathcal N_x$ is a completely prime filter of
  elements of $\Open X$.  $\mathcal U_x$ is clearly upwards-closed in
  the lattice $\mathcal A$ of $\leq_X$-upwards-closed subsets of $X$,
  and closed under arbitrary intersections.  If any union
  $\bigcup_{i \in I} A_i$ of $\leq_X$-upwards-closed subsets of $X$ is
  in $\mathcal U_x$, namely contains $x$, then some $A_i$ contains
  $x$, so $\mathcal U_x$ is completely prime.  \up{For every total
    pair $(U, A)$, namely $U \cup A = X$, we verify that
    $U \in \mathcal N_x$ or $A \in \mathcal U_x$: that simply means
    that $x \in U$ or $x \in A$, which is clear.  For every consistent
    pair $(U, A)$, namely $U \cap A = \emptyset$, we cannot have
    $U \in \mathcal N_x$ and $A \in \mathcal U_x$, namely we cannot
    have $x \in U$ and $x \in A$.}  \down{For every pair $(U, A)$ in
    containment, namely such that $U \supseteq A$, we verify that
    $U \in \mathcal N_x$ or $A \not\in \mathcal U_x$: if
    $A \in \mathcal U_X$, then $x \in A \subseteq U$, so
    $U \in \mathcal N_x$.  If instead $U \subseteq A$, then we verify
    that $U \not\in \mathcal N_X$ or $A \in \mathcal U_x$: if
    $U \in \mathcal N_x$, then $x \in U \subseteq A$, so
    $A \in \mathcal U_x$.}  Hence $(\mathcal N_x, \mathcal U_x)$ is a
  point of $\adO X$.

  For every $U \in \Open X$, ${(\eta_X)}^{-1} (\Open_U) = \{x \in X
  \mid U \in \mathcal N_x\} = U$, so $\eta_X$ is continuous.
  For all $x, y \in X$ such that $x \leq_X y$, $\mathcal U_x \subseteq
  \mathcal U_y$, since every $\leq_X$-upwards-closed set containing
  $x$ must also contain $y$.  By definition of the preordering on
  $\adpt {\adO X}$, $\eta_X (x)$ is smaller than or equal to $\eta_Y (y)$.
\end{proof}

There is a functor $\adpt$ from $\adFrm^{op}$ to $\PreTopcat$.  We
have already defined its action on objects in
Definition~\ref{defn:adpt}.  Given a morphism
$\phi \colon \OmegaL \to \Omega'$ in $\adFrm^{op}$, namely an ad-frame
homomorphism $\phi \eqdef (\varphi, p) \colon \OmegaL' \to \OmegaL$,
we define $\adpt \phi$ as mapping every point $(x, s)$ of $\OmegaL$ to
$(\varphi^{-1} (x), p^{-1} (s))$.  The latter type-checks in the sense
that $x$ is a subset of $\Omega$ and $s$ is a subset of $L$, where
$\OmegaL = (\Omega, L, \up{\tot, \con}, \down{\fof, \cou})$.
\begin{lemma}
  \label{lemma:dpt:func}
  $\adpt$ is a functor from $\adFrm^{op}$ to $\PreTopcat$.
\end{lemma}
\begin{proof}
  Taking the notation above, and letting $\OmegaL'$ be
  $(\Omega', L', \up{\tot', \con'}, \down{\fof', \cou'})$, we first
  verify that $\adpt \phi (x, s)$ is a point of $\OmegaL'$.  Since
  $\varphi^{-1}$ and $p^{-1}$ commute with all unions and
  intersections, $\varphi^{-1} (x)$ is a completely prime filter, and
  $p^{-1} (s)$ is a completely prime complete filter.

  \up{For every total pair $(v, b)$ in $\Omega' \times L'$, we must
    show that $v \in \varphi^{-1} (x)$ or $b \in p^{-1} (s)$.  Since
    $\phi$ preserves totality, $(\varphi (v), p (b))$ is in $\tot'$.
    Since $(x, s)$ is a point, we must then have $\varphi (v) \in x$
    or $p (b) \in s$, which is what we wanted to prove.}
  
  \up{For every consistent pair $(v, b)$ in $\Omega' \times L'$, we
    must show that $v \not\in \varphi^{-1} (x)$ or
    $b \not\in p^{-1} (s)$.  Let us assume that
    $v \in \varphi^{-1} (x)$ and $b \in p^{-1} (s)$.  Then
    $\varphi (v) \in x$ and $p (b) \in s$.  Since $\phi$ preserves
    consistency and since $(v, b) \in \con$, $(\varphi (v), p (b))$ is
    in $\con'$.  But, since $(x, s)$ is a point, we must then have
    $\varphi (v) \not\in x$ or $p (b) \not\in s$, which is
    impossible.}

  \down{For every pair $(v, b) \in \fof'$, we must show that
    $v \in \varphi^{-1} (x)$ or $b \not\in p^{-1} (s)$.  Since
    $(\varphi, p)$ preserves containment, $(\varphi (v), p (b))$ is in
    $\fof$, and since $(x, s)$ is a point, $\varphi (v) \in x$ or
    $p (b) \not\in s$, which is what we wanted to prove.}

  \down{If $(v, b) \in \cou'$ instead, then we rely on the fact that
    $(\varphi, p)$ preserves inclusion, so that
    $(\varphi (v), p (b)) \in \cou$.  Since $(x, s)$ is a point,
    $\varphi (v) \not\in x$ or $p (b) \in s$, so
    $v \not\in \varphi^{-1} (x)$ or $b \in p^{-1} (s)$.}
  
  Next, we need to show that $\adpt \phi$ is continuous and
  order-preserving.  For every open subset $\Open_v$ of
  $\adpt {\OmegaL'}$, where $v \in \Omega'$,
  ${(\adpt \phi)}^{-1} (\Open_v)$ is the collection of points
  $(x, s) \in \adpt {\OmegaL}$ such that
  $\adpt \phi (x, s) \in \Open_v$, namely such that
  $v \in \varphi^{-1} (x)$, namely such that $\varphi (v) \in x$, and
  that is $\Open_{\varphi (v)}$.  Hence $\adpt \phi$ is continuous.
  For all points $(x, s)$ and $(y, t)$ of $\adpt {\OmegaL}$ such that
  $(x, s)$ is below $(y, t)$, namely such that $s \subseteq t$,
  $p^{-1} (s) \subseteq p^{-1} (t)$, so the point
  $\adpt \phi (x, s) = (\varphi^{-1} (x), p^{-1} (s))$ is below
  $(\varphi^{-1} (y), p^{-1} (t)) = \adpt \phi (y, t)$.  Hence
  $\adpt \phi$ is monotonic.

  The fact that $\adpt$ maps the identity morphisms to identity
  morphisms and preserves compositions is obvious.
\end{proof}

\begin{remark}
  \label{rem:semiclosed:pt}
  We have showed that the functor $\adO$ maps \up{upper semi-closed}
  preordered topological spaces to ad-frames satisfying \up{(usc)} and
  \down{lower semi-closed} preordered topological spaces to ad-frames
  satisfying \up{(lsc)} in Remark~\ref{rem:semiclosed}.  Conversely,
  for every ad-frame
  $\OmegaL \eqdef (\Omega, L, \up{\tot, \con}, \down{\fof, \cou})$, we
  claim that if $\OmegaL$ satisfies \up{(usc)}, then
  $X \eqdef \adpt \OmegaL$ is \up{upper semi-closed}.  Let
  $(x, s) \in X$.  Since $s$ is a complete filter of elements of $L$,
  it has a least element $a$ and $s$ is the collection of elements of
  $L$ that are larger than or equal to $a$.  By \up{(usc)}, $a$ is a
  supremum of elements $a'$ such that $(u', a') \in \tot \cap \con$
  for some $u' \in \Omega$.  But $s$ is completely prime, so one of
  those elements $a'$ is in $s$.  Since $a' \leq a$ and $a$ is least
  in $s$, we must conclude that $a=a'$, namely that there is an
  $u \in \Omega$ such that $(u, a) \in \tot \cap \con$.  We claim that
  $\upc_X (x, s)$ is the complement of $\Open_u$, and is therefore
  closed.  Indeed, for every point $(y, t)$,
  $(y, t) \in \upc_X (x, s)$ if and only if $s \subseteq t$, if and
  only if $a \in t$.  Since $(u, a) \in \tot \cap \con$, and by
  definition of points of $\OmegaL$, we have ($u \in y$ or $a \in t$)
  and ($u \not\in y$ or $a \not\in t$), namely, $a \in t$ if and only
  if $u \not\in y$.  Hence $(y, t) \in \upc_X (x, s)$ if and only if
  $u \not\in y$, if and only if $(y, t) \not\in \Open u$.

  Similarly, if $\OmegaL$ satisfies \down{(lsc)}, then $X$ is
  \down{lower semi-closed}.  The proof goes as follows.  Let
  $(x, s) \in X$.  Since $s$ is completely prime, the complement
  $\overline s$ of $s$ is closed under arbitrary suprema; therefore
  $\overline s$ has a largest element $a$.  By \down{(lsc)}, $a$ is an
  infimum of elements $a'$ such that $(u', a') \in \fof \cap \cou$ for
  some $u' \in \Omega$.  If all those elements were in $s$, then $a$
  would be in $s$, too, since $s$ is a complete filter.  Hence some of
  those elements $a'$ is in $\overline s$.  But $a \leq a'$, and since
  $a$ is largest in $\overline s$, $a' \leq a$, so $a=a'$.  Hence $a$
  itself is such that $(u, a) \in \fof \cap \cou$ for some
  $u \in \Open u$.  We claim that $\dc_X (x, s)$ is the complement of
  $\Open_u$, and is therefore closed.  Indeed, for every point
  $(y, t)$, $(y, t) \in \dc_X (x, s)$ if and only if $t \subseteq s$,
  if and only if $\overline s \subseteq \overline t$, if and only if
  $a \not\in t$.  Since $(u, a) \in \fof \cap \cou$, and by definition
  of points of $\OmegaL$, we have ($u \in y$ or $a \not\in t$) and
  ($u \not\in y$ or $a \in t$), namely, $a \in t$ if and only if
  $u \in y$.  Hence $(y, t) \in \dc_X (x, s)$ if and only if
  $u \not\in y$, if and only if $(y, t) \not\in \Open u$.
\end{remark}

\begin{proposition}
  \label{prop:adpt:adj}
  For every preordered topological space $X$, for every ad-frame
  $\OmegaL \eqdef (\Omega, L, \up{\tot, \con}, \down{\fof, \cou})$,
  for every continuous order-preserving map
  $f \colon X \to \adpt {\OmegaL}$, there is a unique ad-frame
  homomorphism $f^! \colon \OmegaL \to \adO X$ such that
  $\adpt {f^!} \circ \eta_X = f$.
\end{proposition}
\begin{proof}
  For every $x \in X$, $f (x)$ is a point of $\OmegaL$, which we will
  write as $(g (x), h (x))$; $g (x)$ is a completely prime filter of
  elements of $\Omega$, and $h (x)$ is a completely prime complete
  filter of elements of $L$.

  If $f^!$ exists, then writing it as $(\varphi, p)$, we must have
  that for every $x \in X$,
  $f (x) = \adpt {f^!} (\eta_X (x)) = \adpt {f^!} (\mathcal N_x,
  \mathcal U_x)$ (where $\mathcal N_x$ is the completely prime filter
  of open neighborhoods of $x$ and $\mathcal U_x$ is the completely
  prime complete filter of $\leq_X$-upwards-closed subsets of $X$
  containing $x$)
  $= (\varphi^{-1} (\mathcal N_x), p^{-1} (\mathcal U_x))$.  Looking
  at first components, we must have
  $g (x) = \varphi^{-1} (\mathcal N_x)$, so for every $x \in X$, for
  every $v \in \Omega$, $v \in g (x)$ if and only if
  $\varphi (v) \in \mathcal N_x$, if and only if $x \in \varphi (v)$.
  Hence $\varphi (v)$ must be equal to $\{x \in X \mid v \in g (x)\}$,
  for every $v \in \Omega$.  Similarly, for every $a \in L$, $p (a)$
  must be equal to $\{x \in X \mid a \in h (x)\}$.  This shows that
  $f^!$ is unique if it exists.

  Let us define $g^! (v)$ as $\{x \in X \mid v \in g (x)\}$ for every
  $v \in \Omega$, $h^! (a)$ as $\{x \in X \mid a \in h (x)\}$ for
  every $a \in L$, and $f^!$ as $(g^!, h^!)$.

  We check that $g^!$ is a frame homomorphism.  For every
  $v \in \Omega$, for every $x \in X$, $x \in g^! (v)$ if and only if
  $v \in g (x)$, by definition of $g^!$, if and only if
  $g (x) \in \Open_v$; so $g^! (v)$ is equal to $g^{-1} (\Open_v)$ for
  every $v \in \Omega$.  But $g^{-1}$ preserves all unions and
  intersections, while $v \mapsto \Open_v$ is a frame homomorphism by
  Lemma~\ref{lemma:adpt}.  Hence $g^!$ is a frame homomorphism.

  Similarly, $h^! (a) = h^{-1} (\Al_a)$ for every $a \in L$, so $h^!$
  is a complete lattice homomorphism, using the fact that
  $a \mapsto \Al_a$ is a complete lattice homomorphism by
  Lemma~\ref{lemma:adpt}.

  \up{Let us check that $f^!$ preserves totality.  For every
    $(v, b) \in \tot$, we must show that $(g^! (v), h^!  (b))$ is
    total, namely that every point $x \in X$ belongs to $g^! (v)$ or
    to $h^! (b)$.  Since $f (x) = (g (x), h (x))$ is a point of
    $\OmegaL$, and since $(v, b) \in \tot$, we must have $v \in g (x)$
    or $b \in h (x)$, namely $x \in g^! (v)$ or $x \in h^! (b)$.}

  \up{Let us check that $f^!$ preserves consistency.  For every
    $(v, b) \in \con$, we must show that $(g^! (v), h^!  (b))$ is
    consistent, namely that $g^! (v)$ and $h^! (b)$ are disjoint.  A
    point in the intersection would be a point $x \in X$ such that
    $v \in g (x)$ and $a \in h (x)$.  Since $f (x) = (g (x), h (x))$
    is a point of $\OmegaL$, and since $(v, b) \in \con$, we must have
    $v \not\in g (x)$ or $b \not\in h (x)$, leading to a
    contradiction.}

  \down{We check that $f^!$ preserves containment.  For every
    $(v, b) \in \fof$, we must show that $g^! (v) \supseteq h^! (b)$.
    For every point $x \in h^! (b)$, $b$ is in $h (x)$.  But
    $f (x) = (g (x), h (x))$ is a point of $\OmegaL$ and
    $(v, b) \in \fof$, so $v \in g (x)$ or $b \not\in h (x)$.  Since
    $b \in h (x)$, we obtain that $v \in g (x)$, hence that
    $x \in g^! (v)$.}

  \down{We check that $f^!$ preserves inclusion.  For every
    $(v, b) \in \cou$, we must show that $g^! (v) \subseteq h^! (b)$.
    For every $x \in g^! (v)$, $v$ is in $g (x)$.  Since
    $f (x) = (g (x), h (x))$ is a point of $\OmegaL$ and
    $(v, b) \in \cou$, we have $v \not\in g (x)$ or $b \in h (x)$.
    But $v \in g (x)$, so $b \in h (x)$, in other words
    $x \in h^! (b)$.}

  Finally, for every $x \in X$,
  $\adpt {f^!} (\eta_X (x)) = ({(g^!)}^{-1} (\mathcal N_x),
  {(h^!)}^{-1} (\mathcal U_x)) = (\{v \in \Omega \mid g^! (v) \in
  \mathcal N_x\}, \{b \in L \mid h^! (b) \in \mathcal U_x\}) = (\{v
  \in \Omega \mid x \in g^! (v)\}, \{b \in L \mid x \in h^!  (b)\}) =
  (\{v \in \Omega \mid v \in g (x)\}, \{b \in L \mid b \in h (x)\}) =
  (g (x), h (x)) = f (x)$.
\end{proof}

We rephrase Proposition~\ref{prop:adpt:adj} as follows.
\begin{theorem}
  \label{thm:adpt}
  There is an adjunction $\adO \dashv \adpt$ between $\PreTopcat$ and
  $\adFrm^{op}$, with unit $\eta$.
\end{theorem}
Theorem~\ref{thm:adpt} is analogue to the Stone adjunction $\Open
\dashv \pt$ between $\Topcat$ and $\Frm^{op}$.

\section{Ad-sobrification}
\label{sec:ad-sobrification}

The usual Stone adjunction $\Open \dashv \pt$ gives rise to a monad
$\pt\Open$, the \emph{sobrification} monad.  For a topological space
$X$, there is a simpler, and naturally isomorphic way of describing
$\pt \Open X$, as the following space $X^s$, the \emph{standard
  sobrification} of $X$, defined as follows.  The points of $X^s$ are
the \emph{irreducible} closed subsets $C$ of $X$, namely the non-empty
closed sets $C$ such that for all closed subsets $C_1$ and $C_2$ of
$X$ such that $C \subseteq C_1 \cup C_2$, $C$ is included in $C_1$ or
in $C_2$.  Equivalently, a closed set $C$ is irreducible if and only
if for every finite family of open sets $U_1$, \ldots, $U_n$, if $C$
intersects each one of them then $C$ intersects
$U_1 \cap \cdots \cap U_n$.  The open subsets of $X^s$ are the sets
$\diamond U \eqdef \{C \in X^s \mid C \cap U \neq \emptyset\}$
\citep[Section~8.2]{JGL-topology}.

Analogously, we call
$\adpt {\adO X}$ the \emph{ad-sobrification} of $X$.  Its points are
the pairs $(\mathcal N, \mathcal U)$ where:
\begin{enumerate}
\item $\mathcal N$ is a completely prime filter of open subsets of
  $X$;
\item $\mathcal U$ is a completely prime complete filter of
  $\leq_X$-upwards-closed subsets of $X$;
\item \up{for every total pair $(U, A)$ consisting of an open subset $U$
  and of an $\leq_X$-upwards-closed subset $A$ of $X$, namely if
  $U \cup A = X$, then $U \in \mathcal N$ or $A \in \mathcal U$;}
\item \up{for every consistent pair $(U, A)$ consisting of an open subset
  $U$ and of an $\leq_X$-upwards-closed subset $A$ of $X$, namely if
  $U \cap A = \emptyset$, then $U \not\in \mathcal N$ or $A \not\in
  \mathcal U$;}
\item \down{for every pair $(U, A)$ consisting of an open subset $U$
  and of an $\leq_X$-upwards-closed subset $A$ of $X$ such that $U
  \supseteq A$, $U \in \mathcal N$ or $A \not\in \mathcal U$;}
\item \down{for every pair $(U, A)$ consisting of an open subset $U$
  and of an $\leq_X$-upwards-closed subset $A$ of $X$ such that $U
  \subseteq A$, $U \not\in \mathcal N$ or $A \in \mathcal U$.}
\end{enumerate}
We can simplify this as follows.

Item~1.  Since $\mathcal N$ is a completely prime filter of open
subsets of $X$, the complement of the largest open subset of $X$ that
is not in $\mathcal N$ is an irreducible closed subset $C$ of $X$, and
$\mathcal N$ is the collection of open subsets of $X$ that intersect
$C$.  In other words, we can represent $\mathcal N$ by $C$, which is a
point in the (standard) sobrification $X^s$ of $X$.

Item~2.  $\mathcal U$ simplifies even more.  Since $\mathcal U$ is a
complete filter, the intersection of all the elements of $\mathcal U$
is still in $\mathcal U$.  In other words, there is an
$\leq_X$-upwards-closed subset $A$ of $X$ such that $\mathcal A$ is
the collection of $\leq_X$-upwards-closed subsets of $X$ containing
$A$.  Additionally, since $\mathcal U$ is completely prime, and since
$A$ is the union of all the subsets $\upc_X x$, $x \in A$ (where
$\upc_X$ denotes upward closure with respect to $\leq_X$), $\upc_X x$
must be in $\mathcal N$ for some $x \in A$.  In other words,
$A \subseteq \upc_X x$ for some $x \in A$.  The reverse inclusion is
obvious, so $A$ is the $\leq_X$-upward-closure of some point
$x \in X$.  (This point is not unique in general, unless $\leq_X$ is
antisymmetric.)  Knowing this, $\mathcal U$ is just the collection of
$\leq_X$-upwards-closed subsets of $X$ that contain $x$.  This point
$x$ is uniquely determined up to $\leq_X \cap \geq_X$, the equivalence
relation associated with $\leq_X$ (equality if $\leq_X$ is
antisymmetric).  Let us write $[x]_X$ for the equivalence class of
$x$.  Then $\mathcal U$ is the collection of $\leq_X$-upwards-closed
subsets of $X$ that contain $x$, for some \emph{unique} equivalence
class $[x]_X$.

\up{Item~3.  The condition on totality then reads: for every open
  subset $U$ of $X$, for every $\leq_X$-upwards-closed subset $A$ of
  $X$, if $U \cup A = X$ then $U$ intersects $C$ or $x \in A$.  We
  claim that this condition is equivalent to the fact that
  $C \cap \dc_X x$ cannot be empty, where $\dc_X$ denotes downward
  closure with respect to $\leq_X$.  If the condition holds, then
  taking $U$ to be the complement of $C$ and $A$ to be the complement
  of $\dc_X x$, we obtain that $U \cap A$ cannot be equal to $X$; by
  taking complements, $C \cap \dc_X x$ cannot be empty.  Conversely,
  if $C \cap \dc_X x$ is non-empty, then let $y$ be in the
  intersection.  For every open subset $U$ of $X$, for every
  $\leq_X$-upwards-closed subset $A$ of $X$, if $U$ does not intersect
  $C$ and $x \not\in A$ then in particular $y$ is not in $U$ and $y$
  is not in $A$ (otherwise, since $y \in \dc_X x$, namely since
  $y \leq_X x$, and since $A$ is $\leq_X$-upwards-closed, $x$ would be
  in $A$), so $U \cup A \neq X$.  This is the contrapositive of the
  condition on totality.}

\up{Item~4.  The condition on consistency reads: for every open subset
  $U$ and every $\leq_X$-upwards-closed subset $A$ of $X$ such that
  $U \cap A = \emptyset$, then $U$ is disjoint from $C$ or
  $x \not\in A$.  We claim that this is equivalent to
  $C \subseteq cl (\upc_X x)$, where $cl$ denotes closure in $X$.
  Indeed, if $C \subseteq cl (\upc_X x)$ failed, then by taking the
  complement of $cl (\upc_X x)$ for $U$ and $\upc_X x$ for $A$, we
  would have $U \cap A = \emptyset$ but $x \in A$, so $U$ would be
  disjoint from $C$ by the condition; but that would say that
  $C \subseteq cl (\upc_X x)$, which is impossible.  Conversely, if
  $C \subseteq cl (\upc_X x)$, then for every open subset $U$ and
  every $\leq_X$-upwards-closed subset $A$ of $X$ such that $U$
  intersects $C$ (hence also $\upc_X x$, say at $y$) and $x \in A$,
  then $U$ would intersect $A$ (at $y$).  This is the contrapositive
  of the condition.}

\down{Item~5.  Symmetrically to item~3, the condition on containment
  is equivalent to the fact that $C \cap \upc_X x \neq \emptyset$,
  where $\upc_X$ is upward closure with respect to $\leq_X$.  Indeed,
  the condition on containment is equivalent to: for every open set
  $U$ and for every $\leq_X$-upwards-closed set $A$ such that
  $U \supseteq A$, $U$ intersects $C$ or $x \not\in A$.  If that
  condition holds, then taking the complement of $C$ for $U$ and
  $\upc_X x$ for $A$, we have $U \cap C = \emptyset$ and $x \in A$, so
  by contraposition $U \not\supseteq A$; in other words,
  $C \cap \upc_X x \neq \emptyset$.  Conversely, if
  $C \cap \upc_X x \neq \emptyset$, then for every open set $U$ and
  for every $\leq_X$-upwards-closed set $A$ such that $U \supseteq A$,
  if $x \in A$, then $\upc_X x \subseteq A \subseteq U$, and since $C$
  intersects $\upc_X x$, it must intersect the larger set $U$.}

\down{Item~6.  Symmetrically to item~4, the condition on inclusion is
  equivalent to $C \subseteq cl (\dc_X x)$.  Indeed, the condition on
  inclusion is equivalent to: for every open set $U$ and for every
  $\leq_X$-upwards-closed set $A$ such that $U \subseteq A$, $U$ is
  disjoint from $C$ or $x \in A$.  If that condition holds, then let
  $U$ be the complement of $cl (\dc_X x)$ and $A$ be the complement of
  $\dc_X x$.  We have $\dc_X x \subseteq cl (\dc_X x)$, so
  $U \subseteq A$, and therefore $U$ is disjoint from $C$ or
  $x \in A$.  But $x$ is not in $A$, so $U$ is disjoint from $C$,
  meaning that $C \subseteq cl (\dc_X x)$.  Conversely,
  $C \subseteq cl (\dc_X x)$, then for every open set $U$ and for
  every $\leq_X$-upwards-closed set $A$ such that $U \subseteq A$, if
  $U$ is not disjoint from $C$, we wish to show that $x \in A$.  Since
  $U$ intersects $C$, $U$ intersects $cl (\dc_X x)$, hence also
  $\dc_X x$.  Let $y \in U \cap \dc_X x$.  Since $U \subseteq A$, $y$
  is in $A$; but $y \leq_X x$ and $A$ is $\leq_X$-upwards-closed, so
  $x \in A$.}

This leads us to the following definition. We recall that $[x]_X$
denotes the equivalence class of $x \in X$ with respect to $\leq_X
\cap \geq_X$.
\begin{definition}
  \label{defn:sobrif}
  The \emph{ad-sobrification} $X^{ads}$ of a preordered topological
  space $X$ is the collection of pairs $(C, [x]_X)$ where $C \in X^s$,
  $x \in X$, \up{$C$ intersects $\dc_X x$,
    $C \subseteq cl (\upc_X x)$}, \down{$C$ intersects $\upc_X x$ and
    $C \subseteq cl (\dc_X x)$}.  Its topology consists of the open
  sets
  $\diamond^{ad} U \eqdef \{(C, [x]_X) \in X^{ads} \mid C \cap U \neq
  \emptyset\}$, and its preordering puts $(C, [x]_X)$ below
  $(D, [y]_X)$ if and only if the collection of
  $\leq_X$-upwards-closed subsets of $X$ that contains $x$ is included
  in the collection of $\leq_X$-upwards-closed subsets of $X$ that
  contains $y$, if and only if $x \leq_X y$.
\end{definition}

Conversely to what we did in the discussion preceding this definition,
any point $(C, [x]_X)$ of $X^{ads}$ yields a point
$(\mathcal N, \mathcal U)$ of $\adpt {\adO X}$, by letting
$\mathcal N$ be the collection of open subsets $U$ of $X$ that
intersect $C$ and $\mathcal U$ be the collection of
$\leq_X$-upwards-closed subsets of $X$ that contain $x$, and the
constructions $(C, [x]_X) \mapsto (\mathcal N, \mathcal U)$ and
$(\mathcal N, \mathcal U) \mapsto (C, [x]_X)$ are inverse of each
other.  Through this bijection, the open subsets $\Open_U$ are
transported to $\diamond^{ad} U$, and the preordering on
$\adpt {\adO X}$ is transported to the preordering given in
Definition~\ref{defn:sobrif}.  In other words,
\begin{fact}
  \label{fact:ads}
  For every preordered topological space, $X^{ads}$ is isomorphic to
  $\adpt {\adO X}$ in $\PreTopcat$.
\end{fact}
Through this isomorphism, the unit
$\eta_X \colon X \to \adpt {\adO X}$ transports (by composition) to a
map which we will again write as $\eta_X$, from $X$ to $X^{ads}$, and
which sends every point $x \in X$ to $(cl (\{x\}), [x]_X)$.  This is
the unit of the monad $\_^{ads} \cong \adpt \adO$ associated with the
adjunction $\adO \dashv \adpt$.  Its action on morphisms is obtained
by tracking the action of $\adpt \adO$ through the isomorphism.  We
make it explicit now.  Let us write $f [C]$ for the image of $C$ under $f$.
\begin{proposition}
  \label{prop:ads:mor}
  For every continuous order-preserving map $f \colon X \to Y$ between
  preordered topological spaces, the functor $\_^{ads}$ applied to $f$
  yields the continuous order-preserving map
  $f^{ads} \colon X^{ads} \to Y^{ads}$ defined by
  $f^{ads} (C, [x]_X) \eqdef (cl (f [C]), [f (x)]_Y)$, for every
  $(C, [x]_X) \in X^{ads}$.
\end{proposition}
\begin{proof}
  For every continuous order-preserving map $f \colon X \to Y$ between
  preordered topological spaces, $\adO f \colon \adO Y \to \adO X$ is
  $(\varphi, p) \eqdef (f^{-1}, f^{-1})$ (see
  Lemma~\ref{lemma:ad:mor}), and then $\adpt {\adO f}$ maps every
  point $(\mathcal N, \mathcal U)$ of $\adO X$ to
  $(\varphi^{-1} (\mathcal N), \allowbreak p^{-1} (\mathcal U))$.
  When $\mathcal N$ is the completely prime filter of open subsets of
  $X$ that intersect $C$, $\varphi^{-1} (\mathcal N)$ is the
  completely prime filter of open subsets $V$ of $Y$ such that
  $\varphi (V) = f^{-1} (V) \in \mathcal N$, namely such that
  $f^{-1} (V)$ intersects $C$, namely such that $V$ intersects
  $f [C]$, or equivalently $cl (f [C])$.  When $\mathcal U = [x]_X$,
  $p^{-1} (\mathcal U)$ is the completely prime complete filter of
  $\leq_Y$-upwards-closed subsets $B$ of $Y$ such that
  $p (B) = f^{-1} (B)$ contains $x$, namely of those that contain
  $f (x)$.
\end{proof}

For every $\leq_X$-upwards-closed subset $A$ of $X$, let
$[A] \eqdef \{(C, [x]_X) \in X^{ads} \mid x \in A\}$.
\begin{lemma}
  \label{lemma:OS}
  For every preordered topological space $X$, the map $\diamond^{ad}$
  is an order-isomorphism of $\Open X$ onto $\Open {X^{ads}}$.  The
  map $A \mapsto [A]$ is an order-isomorphism of the lattice of
  $\leq_X$-upwards-closed subsets of $X$ onto the lattice of
  upwards-closed subsets of $X^{ads}$.
\end{lemma}
\begin{proof}
  $\diamond^{ad}$ is surjective by definition of the topology on
  $X^{ads}$.  It is monotonic: if $U \subseteq V$ in $\Open X$, then
  for every $(C, [x]_X) \in \diamond^{ad} U$, $C$ intersects $U$,
  so $C$ intersects $V$, hence $(C, [x]_X) \in \diamond^{ad} V$.
  We claim that $\diamond^{ad}$ is order-reflecting, namely that
  $\diamond^{ad} U \subseteq \diamond^{ad} V$ implies $U \subseteq V$,
  for all $U, V \in \Open X$.  For every $x \in U$, $cl (\{x\})$
  intersects $U$ (at $x$), so $\eta_X (x) = (cl (\{x\}), [x]_X)$ is
  in $\diamond^{ad} U$.  If
  $\diamond^{ad} U \subseteq \diamond^{ad} V$, then $\eta_X (x)$ must
  also be in $\diamond^{ad} V$, namely $cl (\{x\}$ must intersect $V$.
  Then $\{x\}$ must intersect $V$, so $x \in V$.  As $x$ is arbitrary
  in $U$, $U \subseteq V$.

  Let us write $f$ for the map $A \mapsto [A]$, where $A$ ranges over
  the $\leq_X$-upwards-closed subsets of $X$.  We note that $f (A)$ is
  upwards-closed: if $(C, [x]_X) \in f (A)$, namely if $x \in A$, and
  if $(C, [x]_X)$ is below $(D, [y]_X)$, namely if $x \leq_X y$, then
  $y$ is in $A$, so $(D, [y]_X) \in f (A)$.  It is clear that $f$ is
  monotonic.  Let $A$, $B$ be two $\leq_X$-upwards-closed subsets of
  $X$, and let us assume that $f (A) \subseteq f (B)$.  For every
  $x \in A$, $\eta_X (x)$ is in $f (A)$, hence in $f (B)$, and this
  means that $x \in B$.  Therefore $A \subseteq B$.  We have just
  shown that $f$ is a monotonic order-reflection.
  
  It remains to see that $f$ is surjective.  Let $\Al$ be an
  upwards-closed subset of $X^{ads}$.  We define $A$ as $\{x \in X
  \mid (C, [x]_X) \in \Al \text{ for some }C \in X^s\}$.  $A$ is
  $\leq_X$-upwards-closed:
  if $x \in A$, namely if $(C, [x]_X) \in \Al$ for some $C \in X^s$,
  and if $x \leq_X y$, then $\eta_X (y)$ lies above $(C, [x]_X)$ in
  $X^{ads}$; since $\Al$ is upwards-closed, $\eta_X (y) \in \Al$, so
  $y \in A$.  Finally, $f (A)$ is the collection of points $(D, [x]_X)
  \in X^{ads}$ such that $(C, [x]_X) \in \Al$ for some $C \in X^s$,
  hence it contains $\Al$.  Conversely, every such point $(D, [x]_X)$
  is above (and below) $(C, [x]_X)$ in $X^{ads}$, hence is in $\Al$,
  since $\Al$ is upwards-closed; this shows that $f (A) \subseteq
  \Al$, and therefore that $f (A) = \Al$.
\end{proof}

\begin{lemma}
  \label{lemma:eta}
  For every preordered topological space $X$, for every
  $U \in \Open X$, $\eta_X^{-1} (\diamond^{ad} U) = U$; for every
  $\leq_X$-upwards-closed subset $A$ of $X$, $\eta_X^{-1} ([A]) = A$.
\end{lemma}
\begin{proof}
  $\eta_X^{-1} (\diamond^{ad} U)$ consists of the points $x$ such that
  $(cl (\{x\}), [x]_X) \in \diamond^{ad} U$, namely such that
  $cl (\{x\})$ intersects $U$, or equivalently such that $\{x\}$
  intersects $U$; that is just $U$.  $\eta_X^{-1} ([A])$ consists of
  the points $x$ such that $(cl (\{x\}), [x]_X)$ is such that $[x]_X$
  can be written as $[y]_A$ for some $y \in A$.  Since $A$ is
  $\leq_X$-upwards-closed, hence a union of equivalence classes, this
  is equivalent to $x \in A$.
\end{proof}

In the setting of unordered topological spaces, the unit of the
$\Open \dashv \pt$ adjunction is bijective at a space if and only if
it is a homeomorphism, if and only if $X$ is \emph{sober}, meaning
that every irreducible closed subset is the closure of a unique point.
\begin{deflem}
  \label{deflem:adsober}
  For a preordered topological space $X$, the following are
  equivalent:
  \begin{enumerate}
  \item $\eta_X \colon X \to X^{ads}$ is bijective;
  \item $\eta_X \colon X \to X^{ads}$ is an isomorphism in
    $\PreTopcat$;
  \item $X$ is \emph{ad-sober}, namely every irreducible pair is equal
    to $(cl (\{x\}), [x]_X)$ for a unique point $x$ of $X$.  An
    \emph{irreducible pair} is a pair $(C, [x]_X)$ where $C \in X^s$
    and $x \in X$, such that \up{$C$ intersects $\dc_X x$,
      $C \subseteq cl (\upc_X x)$}, \down{$C$ intersects $\upc_X x$
      and $C \subseteq cl (\dc_X x)$}.
  \end{enumerate}
\end{deflem}
\begin{proof}
  Condition~3 is simply a rephrasing of condition~1, and condition~2
  clearly implies condition~1.  Let us show that 1 implies 2.  We
  assume that $\eta_X$ is bijective, and we let $g$ be its inverse.
  We need to show that $g$ is continuous and monotonic.  By
  Lemma~\ref{lemma:eta}, $\eta_X^{-1} (\diamond^{ad} U) = U$ for every
  $U \in \Open X$; therefore $g^{-1} (U) = \diamond^{ad} U$, showing
  that $g$ is continuous.  Similarly, $g^{-1} (A) = [A]$ for every
  $\leq_X$-upwards-closed subset $A$ of $X$.  Since $[A]$ is
  upwards-closed in $X^{ads}$ (Lemma~\ref{lemma:OS}), $g^{-1} (A)$ is
  upwards-closed for every $\leq_X$-upwards-closed subset $A$ of $X$;
  it is then an easy exercise to deduce that $g$ is monotonic.
\end{proof}

\section{Sobrification and ad-sobrification}
\label{sec:sobr-ad-sobr}

What is the connection between sobrification and ad-sobrification?
There is definitely a similarity, and we make it explicit here.

Let $\Ind$ be the functor that maps every topological space $X$ to the
preordered topological space $X$, with the indiscrete preordering on
$X$ (where every point is below every other).  This acts as identity
on morphisms.
\begin{proposition}
  \label{prop:Ind}
  There is a natural isomorphism between $\_^{ads} \circ \Ind$ and
  $\Ind \circ \_^s$.
\end{proposition}
\begin{proof}
  We investigate objects first.  Let $X$ be a topological space.  If
  $X$ is empty, then $(\Ind X)^{ads}$ is empty, and $X^s$ is empty,
  too.  Hence let us assume that $X$ is non-empty, and let us pick
  $x \in X$.  The equivalence class of $x$ with respect to $\sim$ is
  $X$ itself.  Therefore the points of $(\Ind X)^{ads}$ are the pairs
  $(C, X)$ where $C \in X^s$, \up{$C$ intersects $\dc_X x = X$
    (vacuously true), $C \subseteq cl (\upc_X x) = X$ (also vacuously
    true)}, \down{$C$ intersects $\upc_X x = X$ (vacuously true), and
    $C \subseteq cl (\dc_X x) = X$ (still vacuously true).}  This shows
  that the map $(C, X) \mapsto C$ is a bijection from $(\Ind X)^{ads}$
  onto $X^s$.  The inverse image of the open set $\diamond U$ is
  $\diamond^{ad} U$, so this is a homeomorphism.  Also, $(C, X)$ is
  below $(D, X)$ in $(\Ind X)^{ads}$ is always true, so the
  preordering on $(\Ind X)^{ads}$ is the indiscrete preordering.

  We have obtained an isomorphism $(C, X) \mapsto C$ from
  $(\Ind X)^{ads}$ onto $\Ind {(X^s)}$ for every topological space
  $X$.  (This works when $X$ is empty, too.)

  The naturality property means that for every continuous map
  $f \colon X \to Y$, seen as a continuous order-preserving map from
  $\Ind X$ to $\Ind Y$, for every $(C, [x]_X) \in X^{ads}$, the first
  component of $f^{ads} (C, [x]_X)$ is equal to $f^s (C)$.  But
  $f^s (C) = cl (f [C])$, so the claim follows from
  Proposition~\ref{prop:ads:mor}.
\end{proof}

On the other extreme of the spectrum, let $\Discr$ be the functor that
maps every topological space $X$ to the preordered topological space
$X$ with the discrete preordering, namely equality.  This also acts as
identity on morphisms.  Then ad-sobrification does nothing.
\begin{proposition}
  \label{prop:discr}
  There is a natural isomorphism between $\_^{ads} \circ \Discr$ and
  the identity functor.
\end{proposition}
\begin{proof}
  For every topological space $X$, the points of $(\Discr X)^{ads}$
  are the pairs $(C, x)$ such that \up{$C$ intersects
    $\dc_X x = \{x\}$ (namely, such that $x \in C$),
    $C \subseteq cl (\upc_X x) = cl (\{x\})$}, \down{$C$ intersects
    $\upc_X x = \{x\}$ (equivalent to $x \in C$ again) and
    $C \subseteq cl (\dc_X x) = cl (\{x\})$}.  Those conditions are
  equivalent to requiring that $C = cl (\{x\})$.  Then the map
  $(cl (\{x\}), x) \mapsto x$ is an isomorphism from
  $(\Discr X)^{ads}$ onto $X$, and it is easy to see that it is
  natural in $X$.
\end{proof}

A final construction that is called for consists in seeing every
topological space $X$ as a preordered topological space by giving it
its specialization preordering $\leq$ as preordering.  (The
\emph{specialization preordering} $\leq$ is defined by $x \leq y$ if
and only if every open set that contains $x$ also contains $y$, if and
only if $x$ is in the closure of $\{y\}$.)  Let us write $\dc$ for
downward closure and $\upc$ for upward closure with respect to $\leq$.
Let us also write $\equiv$ for $\leq \cap \geq$.  Then $X^{ads}$ is a
bit weird: it consists of pairs $(C, [x]_X)$ where \up{$C$ intersects
  $\dc x$, $C \subseteq cl (\upc x)$}, \down{$C$ intersects $\upc x$
  and $C \subseteq cl (\dc x) = \dc x$}.  \up{If we only consider the
  red option, then here is an extreme case.  Let us imagine that $X$
  has a least element $\bot$ and a largest element $\top$ with respect
  to $\leq$.  Then $\bot$ is in $C$ and in $\dc x$, and $\top$ is in
  $\upc x$, so $cl (\upc x) = X$ contains $C$, and therefore $C$
  trivially intersects $\dc x$ and the inclusion
  $C \subseteq cl (\upc x)$ is trivial, too.  Therefore, in that case,
  $X^{ads}$ is the product of $X^s$ with $X/\equiv$, with the
  (non-$T_0$) topology whose open subsets are the sets
  $\diamond U \times X/\equiv$, $U \in \Open X$, and whose
  (non-antisymmetric) preordering is given by
  $(C, [x]_X) \leq (D, [y]_X)$ if and only if $x \leq y$}.  \down{If
  we consider the blue-only option, then $C$ intersects $\upc x$ if
  and only if $x \in C$, and the conjunction of this with
  $C \subseteq cl (\dc x) = \dc x$ means that $C = \dc x$.  Then the
  only points of $X^{ads}$ are those of the form $(\dc x, [x]_X)$.
  Similarly to Proposition~\ref{prop:discr}, we can then show that
  $X^{ads}$ is isomorphic to the $T_0$-reflection of $X$, namely to
  $X$ quotiented by $\equiv$, with its specialization ordering, which
  is the quotient of $\leq$ by $\equiv$.}  \both{Finally, in the both
  blue and red option, the only points of $X^{ads}$ are those of the
  form $(C, [x]_X)$ with $C = \dc x$, as in the red-only case, and
  then $C$ trivially intersects $\upc x$ and
  $C = \dc x \subseteq \dc \upc x \subseteq cl (\upc x)$.  Hence
  $X^{ads}$ is also the $T_0$-reflection of $X$ with its
  specialization ordering, as in the red-only case.}

Let us write $|X|$ for the underlying topological space of a
preordered topological space $X$, and let us investigate whether
$|X^{ads}| \cong |X|^s$, where $\cong$ is for homeomorphism.  $|\_|$
extends to a functor $|\_| \colon \PreTopcat \to \Topcat$, which maps
every continuous order-preserving map $f$ to $f$ seen as a mere
continuous map.

\begin{remark}
  \label{rem:|ads|}
  Let $X$ be a preordered topological space.  It is in general wrong
  that $|X^{ads}| \cong |X|^s$.  Indeed, let $X$ be obtained as a
  topological space $|X|$ with the discrete preordering.  By
  Proposition~\ref{prop:discr}, $X^{ads}$ is isomorphic to $X$, hence
  $|X^{ads}| \cong |X|^s$ if and only if $|X| \cong |X|^s$, if and
  only if $|X|$ is sober.  Hence taking any non-sober space for $|X|$
  gives us a counterexample.
\end{remark}

\begin{proposition}
  \label{prop:|ads|}
  For every \up{upper semi-closed} and \down{lower semi-closed}
  pretopological space $X$ such that every irreducible closed subset
  $C$ contains a point $x$ such that \up{$C \subseteq \upc_X x$} and
  \down{$C \subseteq \dc_X x$}, the map $(C, [x]_X) \mapsto C$ is a
  homeomorphism from $|X^{ads}|$ onto $|X|^s$, which is natural in
  $X$.
  
  For every $T_1$ sober pretopological space $X$, the map
  $(C, [x]_X) \mapsto C$ is a homeomorphism from $|X^{ads}|$ onto
  $|X|^s \cong |X|$, which is natural in $X$.
\end{proposition}
\begin{proof}
  Let $\pi$ be the map $(C, [x]_X) \mapsto C$.  For every open subset
  $U$ of $X$, $\pi^{-1} (\diamond U) = \diamond^{ad} U$, so $\pi$ is
  continuous.  This also shows that $\pi$ is a homeomorphism if $\pi$ is
  bijective.

  If $X$ is \up{upper semi-closed}, \down{lower semi-closed}, and
  every irreducible closed subset $C$ contains a point $x$ such that
  \up{$C \subseteq \upc_X x$} and \down{$C \subseteq \dc_X x$}, then
  for every irreducible closed subset $C$ of $X$, we pick such a point
  $x \in C$.  Then \up{$C$ intersects $\dc_X x$,
    $C \subseteq cl (\upc_X x)$}, \down{$C$ intersects $\upc_X x$ and
    $C \subseteq cl (\dc_X x)$}, so $(C, [x]_X)$ is an irreducible
  pair, which is mapped to $C$ by $\pi$.  Hence $\pi$ is surjective.
  If there are two points $x, y \in C$ such that
  \up{$C \subseteq \upc_X x, \upc_X y$} (or
  \down{$C \subseteq \dc_X x, \dc_X y$}), then \up{$x \in \upc_X y$
    and $y \in \upc_X x$, so $y \leq x \leq x$, hence
    $[x]_X = [y]_X$}, \down{and similarly with $\dc_X$ instead of
    $\upc_X$}, showing that $\pi$ is injective.

  If $X$ is $T_1$ and sober then every irreducible closed subset $C$
  of $X$ is of the form $\{x\}$ for a unique point $x$, and then it is
  trivially true that \up{$C$ intersects $\dc_X x$,
    $C \subseteq cl (\upc_X x)$}, \down{$C$ intersects $\upc_X x$, and
    $C \subseteq cl (\dc_X x)$}.  Hence the points of $|X^{ads}|$, as
  well as those of $|X|^s$, are exactly the one-element sets, showing
  that $\pi$ is bijective and that $|X|^s \cong |X|$ at the same time.

  For naturality, let $f \colon X \to Y$ be a monotonic continuous
  map.  Then $f^{ads}$ maps $(C, [x]_X)$ to $(cl (f [C]), [f (x)]_Y)$
  by Proposition~\ref{prop:ads:mor}, so
  $\pi (f^{ads} (C, [x]_X)) = cl (f [C]) = f^s (C) = f^s (\pi (C,
  [x]_X))$.
\end{proof}

\section{The $\adO \dashv \adpt$ adjunction is idempotent}
\label{sec:ad-sobr-adjunct}

The usual adjunction $\Open \dashv \pt$ is idempotent, and we prove a
similar statement for $\adO \dashv \adpt$ here.  This means showing
that $\eta_{\adpt \OmegaL}$ is an isomorphism for every ad-frame
$\OmegaL$, or equivalently that $\adpt \OmegaL$ is ad-sober for every
ad-frame $\OmegaL$, by Definition and Lemma~\ref{deflem:adsober}.
The monad associated with an idempotent adjunction is always an
idempotent monad, where an idempotent monad is one whose
multiplication is a natural isomorphism (there are several equivalent definitions).

\begin{definition}
  \label{defn:T0}
  A preordered topological space $X$ is \emph{ad-$T_0$} if and only if
  the intersection of the specialization ordering $\leq$ of (the
  topological space underlying) $X$ and of the preordering $\leq_X$ is
  antisymmetric.
\end{definition}
Equivalently, $X$ is ad-$T_0$ if and only if for all points $x$ and
$y$ in the same equivalence class with respect to $\equiv_X$ (namely,
$\leq_X \cap \geq_X$) such that $x \neq y$, then there is an open
subset $U$ of $X$ that contains $x$ and not $y$ or that contains $y$
and not $x$.

\begin{lemma}
  \label{lemma:adsober:aux}
  A preordered topological space $X$ is ad-sober if and only if it is
  ad-$T_0$ and its only irreducible pairs are those of the form
  $(cl (\{x\}), [x]_X)$, $x \in X$.
\end{lemma}
\begin{proof}
  Let $X$ be ad-sober, with specialization preordering $\leq$.  Let
  $x$ and $y$ be two points of $X$ such that $x \leq y$, $y \leq x$,
  $x \leq_X y$ and $y \leq_X x$.  Writing $\dc$ for downward closure
  with respect to $\leq$, we know that $cl (\{x\}) = \dc x$.  Then
  $\eta_X (x) = (\dc x, [x]_X) = (\dc y, [y]_X) = \eta_X (y)$.  Both
  $\eta_X (x)$ and $\eta_X (y)$ are in $\adpt X$, hence are
  irreducible pairs, and therefore must be of the form
  $(\dc z, [z]_X)$ for a unique point $z$.  It follows that $x=z=y$.
  Therefore $X$ is ad-$T_0$.

  Conversely, let us assume that $X$ is ad-$T_0$, and that its only
  irreducible pairs are those of the form $(cl (\{x\}), [x]_X)$, or
  equivalently $(\dc x, [x]_X)$.  Then every irreducible pair is of
  that precise form, and we have to show that $x$ is unique.  But if
  $(\dc x, [x]_X) = (\dc y, [y]_X)$, then $x \leq y \leq x$ and
  $x \leq_X y \leq_X x$, so $x=y$ by ad-$T_0$ness.
\end{proof}

\begin{theorem}
  \label{thm:adpt:idemp}
  For every ad-frame
  $\OmegaL \eqdef (\Omega, L, \up{\tot, \con}, \down{\fof, \cou})$,
  $\adpt \OmegaL$ is ad-sober.  The $\adO \dashv \adpt$ adjunction is
  idempotent.  The ad-sobrification monad $\_^{ads}$ is idempotent.
\end{theorem}
\begin{proof}
  We only need to prove the first claim, as the others follow.  Let
  $X \eqdef \adpt \OmegaL$.  We recall (see
  Definition~\ref{defn:adpt}) that its open sets are the sets
  $\Open_u \eqdef \{(x, s) \in \adpt \OmegaL \mid u \in x\}$, where
  $u$ ranges over $\Omega$.  Hence its specialization preordering
  $\leq$ is given by $(x, s) \leq (y, t)$ if and only if for every $u
  \in \Omega$ such that $(x, s) \in \Open_u$, $(y, t) \in \Open_u$, if
  and only if for every $u \in x$, $u$ is in $y$, if and only if $x
  \subseteq y$.  We also recall that the preordering $\leq_X$ is given
  by $(x, s) \leq_X (y, t)$ if and only if $s \subseteq t$.
  Hence, if $(x, s) \leq (y, t)$, $(y, t) \leq (x, s)$, $(x, s) \leq_X
  (y, t)$ and $(y, t) \leq_X (x, s)$, then $x=y$ and $s=t$.  Therefore
  $X$ is ad-$T_0$.

  In order to show that $X$ is ad-sober, we use
  Lemma~\ref{lemma:adsober:aux}.  We consider an arbitrary irreducible
  pair $(C, [(x, s)]_X)$, and we show that it must be of the form
  $(\dc (x_0, s_0), [(x_0, s_0)]_X)$ for some point $(x_0, s_0) \in X$.  Here $\dc$ is
  downward-closure with respect to the specialization ordering $\leq$,
  which is inclusion of first components, as we have just seen.
  Since $[(x, s)]_X$ must be equal to $[(x_0, s_0)]_X$, we define
  $s_0$ as $s$.

  We define $x_0$ as $\bigcup_{(y, t) \in C} y$.  Let us verify that
  $x_0$ is a completely prime filter of elements of $\Omega$.  Being a
  union of completely prime filters, $x_0$ is upwards-closed and
  non-empty.  Showing that $x_0$ is closed under binary infima is a
  bit more complicated.  Let $u, v \in x_0$.  By definition of $x_0$,
  there are elements $(y, t), (y', t') \in C$ such that $u \in y$ and
  $v \in y'$.  Equivalently, $(y, t) \in \Open_u$ and
  $(y', t') \in \Open_v$.  Then $C$ intersects $\Open_u$ (at $(y, t)$)
  and $\Open_v$ (at $(y', t')$).  Since $C \in X^s$, it is
  irreducible, so $C$ intersects $\Open_u \cap \Open_v$, which happens
  to be $\Open_{u \wedge v}$ by Lemma~\ref{lemma:adpt}.  Therefore
  $u \wedge v$ belongs to $y''$ for some $(y'', t'') \in C$.  Since
  $y'' \subseteq x_0$, $u \wedge v \in x_0$.  Hence $x_0$ is a filter.
  Finally, we show that $x_0$ is completely prime.  Let
  ${(u_i)}_{i \in I}$ be an arbitrary family in $\Omega$ whose
  supremum is in $x_0$.  Then that supremum is in $y$ for some
  $(y, t) \in C$, and since $y$ is a completely prime filter, some
  $u_i$ must be in $y$, hence in $x_0$.

  Let us verify that $(x_0, s_0)$ is in $X = \adpt \OmegaL$, namely
  that it is a point of $\OmegaL$.  We must
  check six conditions (see Definition~\ref{defn:adpt}).
  \begin{itemize}
  \item We have just verified the first one, namely that $x_0$ is a
    completely prime filter of elements of $\Omega$.
  \item The second one, that $s_0$ is a completely prime complete
    filter of elements of $L$, is clear since $s_0=s$.
  \item \up{For the third one, let $(u, a) \in \tot$.  We wish to show
      that $u \in x_0$ or $a \in s_0$.  Since $(C, [(x, s)]_X)$ is an
      irreducible pair, $C$ intersects $\dc_X (x, s)$, say at
      $(y, t)$.  Since $(y, t)$ is a point of $\OmegaL$ and
      $(u, a) \in \tot$, we must have $u \in y$ or $a \in t$.  Since
      $(y, t) \in C$, by definition of $x_0$, we have
      $y \subseteq x_0$.  Since $(y, t) \in \dc_X (y, s)$, we have
      $t \subseteq s$.  If $u \in y$, then $y \in x_0$, and if
      $a \in t$, then $a \in s = s_0$.}
  \item \up{For the fourth one, let $(u, a) \in \con$.  We wish to
      show that $u \not\in x_0$ or $a \not\in s_0$.  Let us assume
      that $u \in x_0$ and $a \in s_0$, for the sake of contradiction.
      Since $u \in x_0$, there is a point $(y, t) \in C$ such that
      $u \in y$.  Hence $(y, t) \in \Open_u$.  Since $(C, [(x, s)]_X)$
      is an irreducible pair, $C \subseteq cl (\upc_X (x, s))$, where
      $cl$ denotes closure in $X$.  Hence $\Open_u$ intersects
      $cl (\upc_X (x, s))$, and therefore must also intersect
      $\upc_X (x, s)$.  Let $(y', t')$ be a point in that
      intersection.  Since $(y', t') \in \Open_u$, $u$ is in $y'$.
      Since $(y', t') \in \upc_X (x, s)$, we have
      $(x, s) \leq_X (y', t')$, so $s \subseteq t'$, and because
      $a \in s_0=s$, $a$ is in $t'$.  We have found a pair
      $(u, a) \in \con$ such that $u \in y'$ and $a \in t'$.  This is
      impossible since $(y', t')$ is a point of $\OmegaL$; indeed,
      that contradicts the fourth condition defining points
      (Definition~\ref{defn:adpt}).}
  \item \down{For the fifth one, let $(u, a) \in \fof$.  We wish to
      show that $u \in x_0$ or $a \not\in s_0$.  Equivalently, we
      assume that $a \in s_0$, and we claim that $u \in x_0$.  Since
      $(C, [(x, s)]_X)$ is an irreducible pair, $C$ intersects
      $\upc_X (x, s)$, say at $(y, t)$.  Since $(y, t)$ is a point of
      $\OmegaL$ and $(u, a) \in \fof$, we must have $u \in y$ or
      $a \not\in t$, or equivalently: $a \in t$ implies $u \in y$.  We
      have $a \in s_0 = s$, and $(x, s) \leq_X (y, t)$, so
      $s \subseteq t$, and therefore $a \in t$.  Therefore $u \in y$.
      Since $(y, t) \in C$, by definition of $x_0$, we have
      $y \subseteq x_0$, so $u \in x_0$.}
  \item \down{For the sixth one, let $(u, a) \in \cou$.  We wish to
      show that $u \not\in x_0$ or $a \in s_0$.  Equivalently, we
      assume that $u \in x_0$, and we wish to show that $a \in s_0$.
      Since $u \in x_0$, there is a point $(y, t) \in C$ such that
      $u \in y$.  Hence $(y, t) \in \Open_u$.  Since $(C, [(x, s)]_X)$
      is an irreducible pair, $C \subseteq cl (\dc_X (x, s))$.  Hence
      $\Open_u$ intersects $cl (\dc_X (x, s))$, and therefore also
      $\dc_X (x, s)$.  Let $(y', t')$ be a point in that intersection.
      Since $(y', t') \in \Open_u$, $u \in y'$.  Since
      $(y', t') \in \dc_X (x, s)$, we have $(y', t') \leq_X (x, s)$,
      so $t' \subseteq s$.  We have a pair $(u, a) \in \cou$, and
      $(y', t')$ is a point of $\OmegaL$.  By the sixth condition
      defining points, $u \not\in y'$ or $a \in t'$.  Since
      $u \in y'$, it follows that $a \in t'$, and then that $a \in s$.
      Since $s=s_0$, we obtain the desired conclusion that
      $a \in s_0$.}
  \end{itemize}

  Knowing that $(x_0, s_0)$ is a point of $X$, we show that it lies in
  $C$.  Since $C$ is closed, its complement is an open subset of $X$,
  necessarily of the form $\Open_u$ for some $u \in \Omega$.  If
  $(x_0, s_0)$ were not in $C$, it would be in $\Open_u$, so $u$ would
  be in $x_0$.  By definition of $x_0$, that would imply the existence
  of a point $(y, s) \in C$ such that $u \in y$.  The condition
  $u \in y$ rewrites as $(y, s) \in \Open_u$, which contradicts
  $(y, s) \in C$.

  We use this to show that $C = \dc (x_0, s_0)$.  Since $C$ is closed,
  it is in particular downwards-closed in the specialization
  preordering $\leq$; since $(x_0, s_0) \in C$,
  $\dc (x_0, s_0) \subseteq C$.  For the opposite inclusion, every
  point $(y, s)$ of $C$ is such that $y \subseteq x_0$ by definition
  of $x_0$, namely such that $(y, s) \leq (x_0, s_0)$.

  In summary, $C = \dc (x_0, s_0)$, where $s_0=s$.  Since additionally
  $(x_0, s_0)$ and $(x, s)$ have the same second component,
  $[(x_0, s_0)]_X = [(x, s)]_X$, so
  $(C, [(x, s)]_X) = (\dc (x_0, s_0), [(x_0, s_0)]_X)$.
\end{proof}

\section{The $\adO \dashv \adpt$ adjunction lifts the $\Open \dashv
  \pt$ adjunction}
\label{sec:ado-dashv-adpt-1}

We recall that there is a functor $|\_| \colon \PreTopcat \to \Topcat$
which maps every preordered topological space to its underlying
topological space, and maps every continuous order-preserving map $f$
to $f$ seen as a mere continuous map.

\begin{proposition}
  \label{prop:U:topol}
  The functor $|\_| \colon \PreTopcat \to \Topcat$ is topological.
\end{proposition}
\begin{proof}
  The functor $|\_|$ is faithful, evidently.  It follows that every
  morphism $g \colon |X| \to |Y|$ in $\Topcat$ lifts to at most one
  morphism $\overline g \colon X \to Y$, in the sense that
  $|\overline g|=g$; that happens exactly when $g$ is not only
  continuous but also monotonic with respect to $\leq_X$ and $\leq_Y$.
  
  Let us consider any $|\_|$-source
  ${(g_i \colon E \to |Y_i|)}_{i \in I}$, where each $g_i$ is a
  continuous map.  $E$ is a topological space, while each $Y_i$ is a
  preordered topological space.  We define a preordering $\leq_X$ on
  $E$ by declaring $x \leq_X x'$ if and only if
  $g_i (x) \leq_{Y_i} g_i (x')$ for every $i \in I$, and we call $X$
  the resulting preordered topological space.  (Hence $|X| = E$.)
  Then each $g_i$ lifts to a continuous pre-ordered map
  $\overline g_i \colon X \to Y_i$, in other words $g_i$ is monotonic
  with respect to $\leq_X$ and $\leq_{Y_i}$.  Additionally, for every
  continuous map $h \colon |Z| \to X$, $h$ lifts to a morphism from
  $Z$ to $\overline X$ if and only if $h$ is monotonic with respect to
  $\leq_Z$ and $\leq_X$, if and only if $g_i \circ h$ is monotonic
  with respect to $\leq_Z$ and $\leq_{Y_i}$ for every $i \in I$, if
  and only if $g_i \circ h$ lifts to $Z \to Y_i$ for every $i \in I$.
  Hence the maps $\overline g_i$ form a $|\_|$-initial lift.  We have
  shown that every $|\_|$-source has a (necessarily unique)
  $|\_|$-initial lift.

  For every topological space $E$, the $|\_|$-fiber of $E$ consists of
  those preordered topological spaces $X$ such that $|X| = E$.  This
  can be equated with the collection of preorderings on $E$.  This
  fiber is preordered by $X \leq Y$ if and only if the identity on $E$
  lifts to $X \to Y$.  Seeing $X$ and $Y$ are preorderings, $X \leq Y$
  if and only if $\leq_X$ is included in $\leq_Y$.  The inclusion
  preordering is asymmetric, and this is the property that defines
  $|\_|$ as an amnestic functor.  All these properties make $|\_|$ a
  topological functor.
\end{proof}
Topological functors are always left- and right-adjoint.  The
left-adjoint to $|\_|$ is the discrete functor $\Discr$, while the
right-adjoint is the indiscrete functor $\Ind$.

In particular, topological functors preserve all limits and all
colimits.  They also create all limits and colimits.  It follows
immediately that $\PreTopcat$ is complete and cocomplete.

Mimicking those constructions, there also a functor, which I will
also write as $|\_|$, from $\adFrm$ to $\Frm$, and therefore from
$\adFrm^{op}$ to $\Frm^{op}$.  On objects, it maps every ad-frame
$(\Omega, L, \tot, \con)$ to $\Omega$.  On morphisms, it maps every
ad-frame homomorphism $(\varphi, p)$ to $\varphi$.

It is unlikely that $|\_|$ is topological; in fact it is unlikely that
$|\_|$ is faithful.  But there is an analogue of the indiscrete
functor, which I will write as $\Ind$, just as with (preordered)
topological spaces.  We say that a frame is \emph{trivial} if it
contains just one element (namely, $\bot=\top$), and
\emph{non-trivial} otherwise.  We will restrict to non-trivial frames
below, for simplicity.  The only trivial frame is the one-element
frame, $\adO (\emptyset)$.
\begin{deflem}
  \label{deflem:Ind:loc}
  For every non-trivial frame $\Omega$, let $\Ind (\Omega)$ be
  $(\Omega, \{0, 1\}, \allowbreak \up{\tot_\Ind, \allowbreak
    \con_\Ind}, \down{\fof_\Ind, \allowbreak \cou_\Ind})$ where
  $0 < 1$, 
  \up{the pairs $(u, a)$ in $\tot_\Ind$ are those such that $u$ is the
    largest element $\top$ of $\Omega$ or $a=1$, the pairs
    $(u, a) \in \con_\Ind$ are those such that $u$ is the least
    element $\bot$ of $\Omega$ or $a=0$}, \down{the pairs
    $(u, a) \in \fof_\Ind$ are those such that $u=\top$ or $a=0$, the
    pairs $(u, a) \in \cou_\Ind$ are those such that $u=\bot$ or
    $a=1$}.  Then $\Ind (\Omega)$ is an ad-frame.
\end{deflem}
\begin{proof}
  It is clear that $\{0, 1\}$ is a completely distributive complete
  lattice.  We take the notations $\sqsubseteq$, $\logleq$, etc., from
  Section~\ref{sec:ad-frames}.  Then:
  \begin{itemize}
  \item \up{$\tot_\Ind$ is $\sqsubseteq$-upwards-closed.  If
      $(u, a) \in \tot_\Ind$, then either $u = \top$ and any larger
      pair will have $\top$ as its first component, or $a=1$ and any
      larger pair will have $1$ as its second component.}
  \item \down{$\fof_\Ind$ is $\logleq$-upwards-closed.  If
      $(u, a) \in \fof_\Ind$, then either $u=\top$ and any
      $\logleq$-larger pair will have $\top$ as its first component,
      or $a=0$ and any $\logleq$-larger pair will have $0$ as its
      second component.}
  \item \up{$\tot_\Ind$ contains $\ff \eqdef (\bot, 1)$ and
      $\tt \eqdef (\top, 0)$.}
  \item \down{$\fof_\Ind$ contains $(\bot, 0)$ and $(\top, 1)$.  This
      is clear.}
  \item \up{For all $(u, a), (v, b) \in
      \tot_\Ind$, we wish to show that $(u, a) \logwedge (v, b) \eqdef
      (u \wedge v, a \vee b)$ is in $\tot_\Ind$.  This is true if $a
      \vee b=1$.  Otherwise, $a \vee b=0$, so
      $a=b=0$, hence $u=v=\top$, and therefore $u \wedge v = \top$.}
  \item \down{For all $(u, a), (v, b) \in
      \fof_\Ind$, we wish to show that $(u, a) \sqcap (v, b) = (u
      \wedge v, a \wedge b)$ is in $\fof_\Ind$.  If $a=0$ or
      $b=0$, then $a \wedge
      b=0$, and this is clear.  Otherwise, $u=\top$ and
      $v=\top$, so $u \wedge v=\top$, and this is clear as well.}
  \item \up{For every family of pairs $(u_i, a_i) \in
      \tot_\Ind$, we claim that $\biglogvee_{i \in I} (u_i, a_i)
      \eqdef (\bigvee_{i \in I} u_i, \bigwedge_{i \in I}
      a_i)$ is in $\tot_\Ind$.  This is true if $\bigwedge_{i \in I}
      a_i=1$.  Otherwise, some $a_i$ is equal to
      $0$.  Then the corresponding $u_i$ is equal to
      $\top$, and then $\bigvee_{i \in I} u_i = \top$.}
  \item \down{For every family of pairs $(u_i, a_i) \in
      \fof_\Ind$, we claim that $\bigsqcup_{i \in I} (u_i, a_i) \eqdef
      (\bigvee_{i \in I} u_i, \bigvee_{i \in I}
      a_i)$ is in $\fof_\Ind$.  If $a_i=0$ for some $i \in
      I$, this is clear, since $\bigvee_{i \in I} a_i =
      0$.  Otherwise, $u_i=\top$ for every $i \in I$, so $\bigvee_{i
        \in I} u_i = \top$.}
  \item \up{$\con_\Ind$ is $\sqsubseteq$-Scott-closed.  If $(u, a)
      \sqsubseteq (v, b) \in \con_\Ind$, then either
      $v=\bot$ and $u=\bot$, or $b=0$ and $a=0$, so $(u, a) \in
      \con_\Ind$.  Let ${(u_i, a_i)}_{i \in
        I}$ be a directed family in $\con_\Ind$, $u \eqdef \bigvee_{i
        \in I} u_i$, $a \eqdef \bigvee_{i \in I} a_i$.  If
      $u=\bot$ or if $a=0$, then $(u, a) \in
      \con_\Ind$.  Otherwise, some $u_i$ is different from
      $\bot$ and some $a_j$ is equal to
      $1$.  By directedness, there is an index $k \in
      I$ such that $(u_i, a_i), (u_j, a_j) \sqsubseteq (u_k,
      a_k)$.  Hence $u_k \neq \bot$ and
      $a_k=1$.  But this is impossible since $(u_k, a_k) \in
      \con_\Ind$.}
  \item \down{$\cou_\Ind$ is $\logleq$-Scott-closed.  If
      $(u, a) \logleq (v, b) \in \cou_\Ind$, then either $v=\bot$ and
      then $u=\bot$ since $u \leq v$, or $b=1$ and then $a=1$ since
      $a \geq b$.  For every $\logleq$-directed family
      ${(u_i, a_i)}_{i \in I}$ of elements of $\cou_\Ind$, with
      $\logleq$-supremum $(u, a)$, namely $u = \bigvee_{i \in I} u_i$
      and $a = \bigwedge_{i \in I} a_i$, we claim that $u=\bot$ or
      $a=1$.  Otherwise, $u \neq \bot$, which implies that
      $u_i \neq \bot$ for some $i \in I$, and $a \neq 1$, which
      implies that $a_j \neq 1$ for some $j \in I$.  By directedness,
      there is a $k \in I$ such that
      $(u_i, a_i), (u_j, a_j) \logleq (u_k, a_k)$.  In particular,
      $u_i \leq u_k$ and $a_j \geq a_k$.  Since
      $(u_k, a_k) \in \cou_\Ind$, $u_k = \bot$ or $a_k = 1$.  But
      $u_k=\bot$ would imply $u_i=\bot$, which is impossible, and
      $a_k=1$ would imply $a_j=1$, which is impossible.}
  \item \up{$\con_\Ind$ contains $\ff = (\bot, 1)$ and
      $\tt \eqdef (\top, 0)$.}
  \item \down{$\cou_\Ind$ contains $(\bot, 0)$ and $(\top, 1)$.  This
      is clear.}
  \item \up{For all $(u, a), (v, b) \in \con_\Ind$, we wish to show
      that $(u, a) \logwedge (v, b) = (u \wedge v, a \vee b)$ is in
      $\con_\Ind$.  This is true if $a \vee b=0$.  Otherwise, $a=1$ or
      $b=1$.  If $a=1$ then $u=0$, so $u \wedge v = 0$, and if $b=1$
      then $v=0$ so $u \wedge v=0$, too.}
  \item \down{For all $(u, a), (v, b) \in \cou_\Ind$, we wish to show
      that $(u, a) \sqcap (v, b) = (u \wedge v, a \wedge b)$ is in
      $\cou_\Ind$.  If $a \wedge b = 1$, this is clear.  Otherwise,
      $a \wedge b = 0$, so $a=0$ or $b=0$.  If $a=0$, then $u=\bot$,
      so $u \wedge v = \bot$, and similarly if $b=0$, then $v=\bot$,
      so $u \wedge v = \bot$.}
  \item \up{For every family of pairs $(u_i, a_i) \in \con_\Ind$, we
      claim that
      $\biglogvee_{i \in I} (u_i, a_i) = (\bigvee_{i \in I} u_i,
      \bigwedge_{i \in I} a_i)$ is in $\con_\Ind$.  This is true if
      $\bigwedge_{i \in I} a_i=0$.  Otherwise, every $a_i$ is equal to
      $1$, so every $u_i$ is equal to $\bot$.  Then
      $\bigvee_{i \in I} u_i = \bot$.}
  \item \down{For every family of pairs $(u_i, a_i) \in \cou_\Ind$, we
      claim that
      $\bigsqcup_{i \in I} (u_i, a_i) = (\bigvee_{i \in I} u_i,
      \bigvee_{i \in I} a_i)$ is in $\cou_\Ind$.  This is true if
      $\bigvee_{i \in I} a_i =1$.  Otherwise,
      $\bigvee_{i \in I} a_i=0$, and therefore $a_i=0$ for every
      $i \in I$.  Since $(u_i, a_i) \in \cou_\Ind$, this entails that
      $u_i = \bot$ for every $i \in I$, hence
      $\bigvee_{i \in I} u_i = \bot$.}
  \item The interaction laws.
    \begin{itemize}
    \item \up{Let $(u, a) \in \con_\Ind$ and $(v, b) \in \tot_\Ind$.
        We wish to show that if $u=v$ then $a \leq b$, and that if
        $a=b$ then $u \leq v$.  We first assume $u=v$.  We have
        $u=\bot$ or $a=0$ (since $(u, a) \in \con_\Ind$), and $v=\top$
        or $b=1$ (since $(v, b) \in \tot_\Ind$).  If $a \leq b$
        failed, we would have $a=1$ and $b=0$, hence $u = \bot$ and
        $v = \top$, contradicting $u=v$.  Second, we assume $a=b$.  If
        $a=b=0$, then since $(v, b) \in \tot_\Ind$, $v=\top$ and hence
        $u \leq v$, trivially.  If $a=b=1$, then since
        $(u, a) \in \con_\Ind$, $u=\top$ and therefore $u \leq v$,
        trivially.}
    \item \down{Let $(u, a) \in \cou_\Ind$ and $(v, b) \in \fof_\Ind$.
        We wish to show that if $u=v$ then $a \geq b$, and that if
        $a=b$ then $u \leq v$.  We first assume $u=v$.  If $a \geq b$
        failed, we would have $a=0$ and $b=1$.  But since
        $(u, a) \in \cou_\Ind$, $a=0$ would entail $u = \bot$, and
        since $(v, b) \in \fof_\Ind$, $b=1$ would entail $v = \top$,
        and this contradicts $u=v$.  Next, we assume that $a=b$.  If
        $a=b=0$, then $(u, a) \in \cou_\Ind$ entails that $u=\bot$, so
        $u \leq v$.  If $a=b=1$, then $(v, b) \in \fof_\Ind$ entails
        that $v=\top$, so $u \leq v$.}
    \item \both{Let $(u, a) \in \con_\Ind \cap \cou_\Ind$.  Then
        $u=\bot$ or $a=0$, and $u=\bot$ or $a=1$.  Whether $a$ is
        equal to $0$ or $1$, we must have $u=\bot$.}
    \item \both{Let $(v, b) \in \tot_\Ind \cap \fof_\Ind$.  Then
        $v=\top$ or $b=1$, and $v=\top$ or $b=0$.  In any case,
        $v=\top$.}
    \item \both{Let $(v, b) \in \con_\Ind \cap \fof_\Ind$.  Then
        $v=\bot$ or $b=0$, and $v=\top$ or $b=0$.  In any case, $b=0$.
        Indeed $b \neq 0$ would entail
        $v=\bot$ and $v=\top$, which is impossible since $\Omega$ is non-trivial.}
    \item \both{Let $(u, a) \in \tot_\Ind \cap \cou_\Ind$.  Then
        $u=\top$ or $a=1$, and $u=\bot$ or $a=1$.  In any case, $a=1$.
        Indeed $a \neq 1$ would entail $u=\top$ and $u=\bot$, which is
        impossible since $\Omega$ is non-trivial.}
    \end{itemize}
  \end{itemize}
\end{proof}

Let $\Frm_*$ be the full subcategory of $\Frm$ consisting of
non-trivial frames.  Let us say that an ad-frame
$\OmegaL \eqdef (\Omega, L, \up{\tot, \con}, \down{\fof, \cou})$ is
\emph{non-trivial} if and only if $\Omega$ is non-trivial.  We let
$\adFrm_*$ be the full subcategory of $\adFrm$ consisting of
non-trivial ad-frames.

\begin{deflem}
  \label{deflem:Ind:loc:mor}
  There is a functor $\Ind \colon \Frm_* \to \adFrm_*$ that is defined on
  objects as in Definition and Lemma~\ref{deflem:Ind:loc}, and maps
  every frame homomorphism $\psi \colon \Omega \to \Omega'$ to
  $(\psi, \identity {\{0, 1\}})$.
\end{deflem}
\begin{proof}
%
  We need to show that $(\psi, \identity {\{0, 1\}})$ is an ad-frame
  homomorphism.  Clearly, $\psi$ is a frame homomorphism and
  $\identity {\{0, 1\}}$ is a complete lattice homomorphism.  \up{For
    all $u \in \Omega$ and $a \in \{0, 1\}$, if $u=\top$ or $a=1$ then
    $\psi (u)=\top$ (since $\psi$ maps $\top$ to $\top$) or $a=1$,
    showing that $(\psi, \identity {\{0, 1\}})$ preserves totality; if
    $u=\bot$ or $a=0$ then $\psi (u)=\bot$ (since $\psi$ maps $\bot$
    to $\bot$) or $a=0$, showing that $(\psi, \identity {\{0, 1\}})$
    preserves consistency.}  \down{If $u=\top$ or $a=0$, then
    $\psi (u)=\top$ or $a=0$, showing that
    $(\psi, \identity {\{0, 1\}})$ preserves containment.  If $u=\bot$
    or $a=1$, then $\psi (u)=\bot$ or $a=1$, showing that
    $(\psi, \identity {\{0, 1\}})$ preserves inclusion.}


  The fact that $\Ind$ preserves identities and composition is trivial.
\end{proof}


\begin{deflem}
  \label{deflem:epsilon:loc}
  For every non-trivial ad-frame
  $\OmegaL \eqdef (\Omega, L, \up{\tot, \con}, \allowbreak
  \down{\fof, \cou})$, there is an ad-frame homomorphism
  $\epsilon_{\OmegaL} \colon \Ind {|\OmegaL|} \to \OmegaL$ defined as
  $(\identity \Omega, bnd_L)$, where $bnd_L \colon \{0, 1\} \to L$
  maps $0$ to the least element $\bot$ of $L$ and $1$ to the largest
  element $\top$ of $L$.
\end{deflem}
\begin{proof}
  The identity map $\identity \Omega$ is a frame homomorphism.  It is
  clear that $bnd_L$ is a complete lattice homomorphism.

  \up{Let us show that $\epsilon_{\OmegaL}$ preserves totality.  Let
    $(u, a)$ be a total pair in $\Ind {|\OmegaL|}$, namely $u=\top$ or
    $a=1$.  Then $\epsilon_{\OmegaL} (u, a) = (u, bnd_L (a))$.  If
    $a=1$, then the latter is equal to $(u, \top)$.  But
    $\ff \eqdef (\bot, \top)$ is in $\tot$, and $\tot$ is
    $\sqsubseteq$-upwards-closed, where $\sqsubseteq$ is the
    information ordering $\leq \times \leq$.  Hence
    $\ff \sqsubseteq (u, \top)$, so $(u, \top) \in \tot$.  If $a=0$,
    then $u=\top$, so
    $\epsilon_{\OmegaL} (u, a) = (\top, \bot) = \tru$, which is in
    $\tot$.  In both cases, $\epsilon_{\OmegaL} (u, a) \in \top$.}

  \up{We now show that $\epsilon_{\OmegaL}$ preserves consistency.
    Let $(u, a)$ be a consistent pair in $\Ind {|\OmegaL|}$, namely
    $u=\bot$ or $a=0$.  Then
    $\epsilon_{\OmegaL} (u, a) = (u, bnd_L (a))$.  If $a=1$, then
    $u=\bot$, and the latter is equal to $(\bot, \top) = \ff$, which
    is in $\con$.  If $a=0$, then
    $\epsilon_{\OmegaL} (u, a) = (u, \bot) \sqsubseteq (\top, \bot) =
    \tt$.  Since $\tt \in \con$ and $\con$ is Scott-closed, in
    particular downwards closed with respect to $\sqsubseteq$,
    $\epsilon_{\OmegaL} (u, a) \in \con$.}

  \down{Next, we show that $\epsilon_{\OmegaL}$ preserves containment.
    Let $(u, a)$ be a pair related by containment in
    $\Ind {|\OmegaL|}$, namely $u = \top$ or $a=0$.  We wish to show
    that $(u, bnd_L (a)) \in \fof$.  If $a=0$, then
    $(u, bnd_L (a)) = (u, \bot)$.  Since $(\bot, \bot) \in \fof$ and
    $\fof$ is $\logleq$-upwards-closed (where $\logleq$ is the logical
    ordering $\leq \times \geq$), $(u, \bot)$ is also in $\fof$.  If
    $a=1$, then $u=\top$, so $(u, bnd_L (a)) = (\top, \top)$, which is
    in $\fof$.}

  \down{Finally, we show that $\epsilon_{\OmegaL}$ preserves
    inclusion.  Let $(u, a)$ be a pair related by inclusion in
    $\Ind {|\OmegaL|}$, namely $u=\bot$ or $a=1$.  If $a=0$, then
    $u=\bot$, so $(u, bnd_L (a)) = (\bot, \bot)$, which is in $\cou$.
    If $a=1$, then $(u, bnd_L (a)) = (u, \top)$.  But
    $(u, \top) \logleq (\top, \top) \in \cou$, and $\cou$ is
    $\logleq$-Scott-closed, hence $\logleq$-downwards-closed, whence
    $(u, \top) \in \cou$.}
\end{proof}

\begin{remark}
  \label{rem:bnd}
  The map $bnd_L \colon \{0, 1\} \to L$ is the \emph{unique} complete
  lattice homomorphism from $\{0, 1\}$ to $L$, since any such
  homomorphism must map $0$ to $\bot$ and $1$ to $\top$.
\end{remark}

\begin{proposition}
  \label{prop:Ind:adj:loc}
  For every non-trivial frame $\Omega'$, for every non-trivial
  ad-frame
  $\OmegaL \eqdef (\Omega, L, \up{\tot, \con}, \down{\fof, \cou})$,
  for every ad-frame homomorphism
  $(\varphi, p) \colon \Ind (\Omega') \to \OmegaL$, there is a unique
  frame homomorphism $\psi \colon \Omega' \to |\OmegaL|$ such that
  $\epsilon_{\OmegaL} \circ \Ind (\psi) = (\varphi, p)$.

  Hence there is an adjunction
  $\adFrm_*^{op} \colon|\_| \dashv \Ind \colon \Frm_*^{op}$, with unit
  $\epsilon$.
\end{proposition}
\begin{proof}
%
%
  Let us turn to the existence and uniqueness of $\psi$.  The
  condition $\epsilon_{\OmegaL} \circ \Ind (\psi) = (\varphi, p)$
  means that $(\psi, bnd_L) = (\varphi, p)$.  Uniqueness follows
  immediately.  In order to show existence, we first note that
  $p=bnd_L$ by Remark~\ref{rem:bnd}, and we define $\psi$ as
  $\varphi$.

  The last part of the proposition follows by reversing arrows.
\end{proof}

Let $\Topcat_*$ be the full subcategory of $\Topcat$ consisting of
non-empty topological spaces, and $\PreTopcat_*$ be the full
subcategory of $\PreTopcat$ consisting of non-empty preordered
topological spaces.
\begin{theorem}
  \label{thm:lift}
  The following square commutes up to isomorphism:
  \[
    \xymatrix{
      \PreTopcat_*
      \ar[r]^{\adO} \ar@{}[r]
      &
      \adFrm_*^{op}
      \\
      \Topcat_*
      \ar[r]^{\Open} \ar@{}[r]
      \ar[u]_{\Ind}
      & \Frm_*^{op}
      \ar[u]_{\Ind}
    }
  \]
\end{theorem}
\begin{proof}
  Starting from a non-empty topological space $X$, $\adO {\Ind X}$ is
  the non-trivial ad-frame
  $(\Open X, \{\emptyset, \allowbreak X\}, \allowbreak \tot, \con)$,
  where \up{the pairs $(U, A) \in \Open X \times \{\emptyset, X\}$
    that are in $\tot$ are those such that $U \cup A = X$, those in
    $\con$ are those such that $U \cap A = \emptyset$}, \down{those in
    $\fof$ are those such that $U \supseteq A$ and those in $\cou$ are
    those such that $U \subseteq A$}.  We note that \up{$U \cup A = X$
    if and only if $U=X$ or $A=X$: if $A \neq X$, then $A$ is empty,
    and then $U \cup A = X$ simplifies to $U=X$.  Also,
    $U \cap A = \emptyset$ if and only if $U=\empty$ or $A=\emptyset$:
    if $A \neq \emptyset$, then $A=X$, and $U \cap A = \emptyset$
    simplifies to $U=\emptyset$.}  \down{Next, $U \supseteq A$ if and
    only if $U=X$ or $A=\emptyset$, and $U \subseteq A$ if and only if
    $U=\emptyset$ or $A=X$.}

  Hence
  $(\Open X, \{\emptyset, X\}, \up{\tot, \con}, \down{\fof, \cou})$ is
  isomorphic to
  $(\Open X, \{0, 1\}, \up{\tot_\Ind, \allowbreak \con_\Ind},
  \allowbreak \down{\fof_\Ind, \cou_\Ind})$, namely to
  $\Ind (\Open X)$.  This shows that the composition of the southwest
  $\to$ northwest $\to$ northeast arrows coincides, up to isomorphism,
  with the composition of the southwest $\to$ southeast $\to$
  northeast arrows.
\end{proof}

\begin{remark}
  \label{rem:lift}
  One may wonder whether one would have a commuting square of arrows
  going in the other direction, namely $|\adpt| \cong \pt |\_|$.  That
  would mean that $|\adpt \OmegaL|$ is (naturally) isomorphic to
  $\pt |\OmegaL|$ for every ad-frame $\OmegaL$, but this would be
  wrong.  This is already wrong when $\OmegaL = \adO X$ for a
  preordered topological space $X$.  Indeed, in that case
  $|\adpt \OmegaL| \cong |X^{ads}|$, while
  $\pt |\OmegaL| \cong |X|^s$, and any non-sober space with the
  discrete preordering will be a counterexample by
  Remark~\ref{rem:|ads|}.
\end{remark}

\section{Conclusion}
\label{sec:conclusion}

We have obtained an adjunction $\adO \dashv \adpt$ between
$\PreTopcat$ and $\adFrm^{op}$, or rather three, \up{red}, \down{blue}
and \both{both}.  Just like van der Schaaf's
\citep{vdS:ordered:locales, HvdS:ordered:locales}, it lifts the usual
$\Open \dashv \pt$ adjunction, under the proviso that we restrict it
to non-empty topological spaces and non-trivial ad-frames.  This
proviso is a minor defect of the theory, and one that I think we can
live with.

The adjunction $\adO \dashv \adpt$  shares many nice properties with
the $\Open \dashv \pt$ adjunction.  Notably, we have shown that there
is a notion of ad-sobrification, which is idempotent (equivalently, the
monad is idempotent).

Based on Remark~\ref{rem:semiclosed} and
Remark~\ref{rem:semiclosed:pt}, $\adO \dashv \adpt$ specializes to an
adjunction between \up{upper semi-closed} (resp.\ \down{lower
  semi-closed}, resp.\ \both{semi-closed}) preordered topological
spaces and ad-frames satisfying property \up{(usc)}, resp.\
\down{(lsc)}, resp.\ \both{both}.  One part of the theory that bothers
me is that it does not seem that we can characterize \emph{pospaces}
this way, that is, ordered topological spaces $X$ with an ordering
whose graph is closed in $X \times X$.  In a sense, ad-frames miss a
way of speaking of the ordering relation $\leq$ itself, and only has
access to upwards-closed and downwards-closed subsets instead.

One thing that we have not developed here is a theory of
\emph{ad-locales}.  Just like locales are just frames, with morphisms
reversed, ad-locales are just ad-frames with morphisms reversed.
There is an extensive program of translating known theorems on ordered
topological spaces to analogous theorems on ad-locales, and in the
process, making them \emph{pointfree}.  The hope is that the pointfree
analogues of known theorems would be constructive, paralleling a
program that has been set up and followed by Isbell, Johnstone,
Banaschewski, Escard\'o, and others in the unordered setting.  One
may for example look at results of \citet{Lawson:chains} on chains in
partially ordered spaces for a start.

One may also hope that a theory of ad-locales would contain some new
theorems, with no analogue in the non-pointfree setting, such as
Isbell's density theorem \citep{Isbell:density}.

My own interest in the field is one of curiosity, and I would like to
extend the theory developed here to finding Stone-like dualities for
\emph{streams}, a model of topological spaces with \emph{local}
preorderings \citep{Krishnan:convenient}.

Finally, Stone-like dualities are deeply connected to logic, and one
should investigate what the logic of ad-frames is.  The natural
inspiration would be to follow \citet[Chapter~6]{Jakl:dframes}.  One
should also investigate connections with intuitionistic S4 logic
\citep{WZ:int:modal}, where intuitionism would be about the topology
and the S4 modalities would be about the ordering.

\section{Acknowledgments}
\label{sec:acknowledgments}

Thanks to Nesta van der Schaaf for discussing ordered locales with me.
We also talked about timespaces in general relativity; since I know
essentially nothing about them, I refrained from mentioning them in
this paper, and especially why pointless timespaces are important
there: see Nesta's PhD thesis \citep{vdS:ordered:locales}.



\DeclareRobustCommand{\VAN}[3]{#3}

\bibliographystyle{apalike}
\bibliography{pspdual}
\end{document}